\documentclass[a4paper,12pt]{article}
\usepackage{latexsym}
\usepackage{mathrsfs}
\usepackage{amsmath,amsthm}
\usepackage{graphicx}
\usepackage{amssymb}
\usepackage{CJK}
\usepackage{color}
\newtheorem{thm}{Theorem}[section]
\newtheorem{cor}[thm]{Corollary}
\newtheorem{lem}[thm]{Lemma}
\newtheorem{prop}[thm]{Proposition}

\newtheorem{alg}[thm]{Algorithm}
\DeclareGraphicsExtensions{.eps,.eps.gz}
\normalsize \rm
\usepackage{geometry}
\usepackage[centerlast]{caption}

\title{Characterizing the fullerene graphs with the minimum forcing number $3$\footnote{This work is supported by NSFC (grant no. 11871256) and the Fundamental Research Funds for the Central Universities (grant no. lzujbky-2017-28).}}
\author{Lingjuan Shi$^{a,b}$, Heping Zhang$^a$\thanks{Corresponding author.}, Ruizhi Lin$^a$ }
\date{\small $^a$School of Mathematics and Statistics, Lanzhou University, Lanzhou, Gansu 730000, \\P. R. China\\
$^{b}$School of Software and Microelectronics, Northwestern Polytechnical University, \\Xi'an, Shaanxi 710072, P. R. China\\
E-mails: shilj15@lzu.edu.cn, zhanghp@lzu.edu.cn, linrzh08@lzu.edu.cn}

\begin{document}
\maketitle
\begin{abstract}
The minimum forcing number of a graph $G$ is the smallest number of edges simultaneously contained in a unique
perfect matching of $G$. Zhang, Ye and Shiu \cite{HDW} showed that the minimum forcing number of any fullerene graph was bounded below by $3$. However, we find that there exists exactly one excepted fullerene $F_{24}$ with the minimum forcing number $2$.
In this paper, we characterize all fullerenes with the minimum forcing number $3$ by a construction approach. This  also solves an open problem proposed by Zhang et al.  We also find that except for $F_{24}$, all fullerenes with anti-forcing number $4$ have the minimum forcing number $3$. In particular, the nanotube fullerenes of type $(4, 2)$ are such fullerenes.\\

\end{abstract}
\section{Introduction}
A \emph{fullerene} graph (simply fullerene) is a cubic $3$-connected plane graph with only pentagonal and hexagonal
faces. By Euler's formula, a fullerene graph has exactly twelve pentagonal faces.
Such graphs are suitable models for carbon fullerene molecules: carbon atoms are
represented by vertices, whereas edges represent chemical bonds between two atoms (see \cite{fowler1995, LZ18}).
Gr\"{u}nbaum and Motzkin \cite{Gru} showed that a fullerene graph with $n$ vertices exists for $n=20$ and for all even $n\geq24$.

A \emph{perfect matching} $M$ of a graph $G$ is an  edge set such that each vertex of $G$ is   incident with exactly one edge in $M$.
Let $G$ be a graph with a perfect matching $M$. A set $S\subseteq M$ is called a \emph{forcing set} of
$M$ if $S$ is not contained in any other perfect
matchings of $G$.  The \emph{forcing
number}  of $M$,  first proposed in organic chemistry by Randi\'c and Klein \cite{DM,RK,MR} under name  {\em innate degree of freedom} in correlation with resonance structure, is defined as the minimum size of all forcing
sets of $M$ by Harary et al. \cite{FDT}, denoted by $f(G, M)$. The \emph{minimum forcing number} of $G$, denoted by $f(G)$, is the minimum
value of the forcing numbers of all perfect matchings of $G$.
\begin{figure}[htbp!]
\centering
\includegraphics[height=2.5cm]{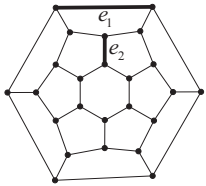}
\caption{\label{F_24}{\small $\{e_1, e_2\}$ is a minimum forcing set of fullerene $F_{24}$.}}
\end{figure}

D. Vuki\v{c}evi\'{c} and N.
Trinajsti\'{c} \cite{afcben,afcata} recently introduced the anti-forcing number of a graph $G$ as the smallest number of edges whose removal results in a subgraph with a single perfect matching,  denoted  by $af(G)$.
For a fullerene graph $F$, Yang et al. \cite{YangQ} showed that $af(F)\geq4$, and further
gave a procedure to construct all fullerenes with the anti-forcing number $4$.
For a $(3, 6)$-fullerene graph $H$, two of the present authors   proved \cite{SZfullerene}that $af(H)\geq2$ and equality holds if and only if $H$ either has connectivity $2$ or is isomorphic to $K_4$, and determined all the $(3, 6)$-fullerenes with the anti-forcing number $3$. Jiang and Zhang \cite{BN-fullerene} characterized all the $(4, 6)$-fullerenes with the minimum forcing number $2$.

In Ref. \cite{HDW}, Zhang, Ye and Shiu gave a main result that the  forcing number of of every perfect matching of any fullerene graph was bounded below by $3$ (See Theorem 2.7).
We find that in the last paragraph of its proof, they neglected the trivial case of cyclic 5-edge cut $S$, that is, the claim ``Clearly, $S$ is non-trivial'' is not right. In fact,  when $S$ is trivial, we obtain a unique fullerene $F_{24}$ with the minimum forcing number $2$ (see Fig. \ref{F_24}).  So the main result (Theorem 1.1 or 2.7) in  \cite{HDW} can be corrected as

\begin{thm}Let $F$ be a fullerene graph. Then $f(F) \geq  3$ except for $F_{24}$.
\end{thm}

We also should point out the fact by Yang et al. \cite{YangQ}  that the anti-forcing number of any fullerene graphs are at least 4  still holds, although they applied the wrong lower bound $3$ of the minimum forcing number. This is because  $F_{24}$ has the anti-forcing number $4$, and all the discussions in that fact are not affected if we exclude the fullerene $F_{24}$.

In this paper, we focus on studying properties of fullerenes with the minimum forcing number $3$. By  applying these properties we obtain a procedure to generate all fullerenes with the minimum forcing number $3$.
Hence we give a solution to an open Problem 4.1 proposed by Zhang et al. \cite{HDW}.

\section{Preliminaries}
For a graph $G$ with $\emptyset\not=X\subset V(G)$, let $\partial X$
be the set of edges with only one end in $X$. Then $\partial X$ is an \emph{edge-cut} of $G$.
For a subgraph $H$ of $G$ with $V(H)\neq V(G)$, we denote by $\triangledown_G(H)$ the edge set of $G$ with only one end in $V(H)$.
An edge-cut
$C$ of $G$ is \emph{trivial} if its edges are all incident with the same vertex, and \emph{cyclic} if $G-C$
has at least two components, each containing a cycle.
A cyclic $k$-edge-cut (with $k$ edges) of $G$ is \emph{trivial} if at least one of the two components
is a single cycle of length $k$.
The \emph{cyclic edge-connectivity} of $G$ is the minimum size of cyclic edge-cuts of $G$
\begin{figure}[htbp!]
\centering
\includegraphics[height=3.2cm]{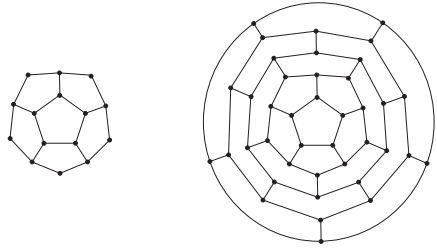}
\caption{\label{G^k}{\small A pentacap (left) and $G^2$(right).}}
\end{figure}
\begin{lem}[\cite{Doslic, YangQ}]\label{3-4-edge-cut}
For a fullerene graph,  every $3$-edge-cut is trivial, every $4$-edge-cut isolates an edge, and the cyclic edge-connectivity is $5$.
\end{lem}

We denote by $G^k$ the tubular fullerene graph comprised of two pentacaps and $k$ layers of hexagons between them. For example,  see Fig. \ref{G^k} for $k=2$.
\begin{thm}[\cite{FR, Kutnar2008}]
A fullerene graph has a non-trivial cyclic $5$-edge-cut if and only if it is
isomorphic to the graph $G^k$ for some integer $k\geq1$.
\end{thm}

From the proof of the theorem or Ref. \cite{LZ18} we  easily have the following lemma.
\begin{lem}\label{component-of-non-trivial-cyclic-5-edge-cut}
Let $S$ be a non-trivial cyclic $5$-edge-cut of fullerene $G^k$. Then $G^k-S$ has two components $H_1$ and $H_2$, and $H_i$ consists of a pentacap and $l$ \emph{(}$0\leq l\leq k-1$\emph{)} layers of hexagons around it. Moreover, for any $2$-degree vertex $x$ in $H_i$, $H_i-x$ is $2$-connected.
\end{lem}

A cyclic edge-cut $C$ of a fullerene graph $F$ is \emph{non-degenerate} if both components of $G-C$ contain precisely six pentagons, and \emph{degenerate} otherwise. So the trivial cyclic edge-cuts of $F$ are degenerate and the non-trivial cyclic $5$-edge-cuts of $F$ are non-degenerate.

A \emph{patch} of a fullerene graph $F$ is a $2$-connected subgraph of $F$ whose all interior
faces are faces of $F$ and all vertices not on the outer face have degree $3$.
A \emph{generalized patch} of $F$ is a connected plane subgraph where all interior faces (if exists) are
faces of $F$, and vertices not on the outer face have
degree $3$ and vertices on the outer face have degree $1$, $2$ or $3$. Clearly, a patch is also a generalized patch.

By Theorem $3$ in \cite{FR}, we have the following lemma.
\begin{lem}\label{5-pentagons-cylic-6-edge-cut}
Let $H$ be a patch of a fullerene $F$ and $\triangledown_F(H)$ a degenerate cyclic $6$-edge-cut of $F$ such that $H$ has at most five pentagons. If $H$ has eaxctly five pentagons or $H$ has at least $14$ vertices, then $H\cong P^1$ or $P^2$ as shown in Fig. \emph{\ref{sub-patches}}.
\end{lem}
\begin{figure}[htbp!]
\centering
\includegraphics[height=2.4cm]{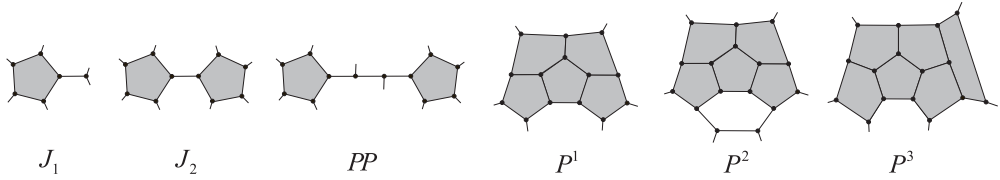}
\caption{\label{sub-patches}{\small Six generalized patches of fullerenes.}}
\end{figure}

From Ref. \cite{cyclic-7}, we have the following.
\begin{lem}\label{chooseD}
Let $H$ be a generalized patch of a fullerne $F$ with at most five pentagons and $\triangledown_F(H)$ be a degenerate cyclic $7$-edge-cut of $F$. Then $H$ is isomorphic to one of the patches $D_{01}$, \ldots, $D_{57}$ as shown in Fig. \emph{7} in Ref. \emph{\cite{cyclic-7}}. Moreover, if $H$ is $2$-connected and has a $2$-degree vertex $x$ such that $H-x$ has a unique perfect matching, then $H$ is one of the patches as depicted in Fig. \emph{\ref{patches}}, and such $2$-degree vertices are the white vertices as shown in Fig. \emph{\ref{patches}}.
\end{lem}
\begin{figure}[htbp!]
\centering
\includegraphics[height=5.6cm]{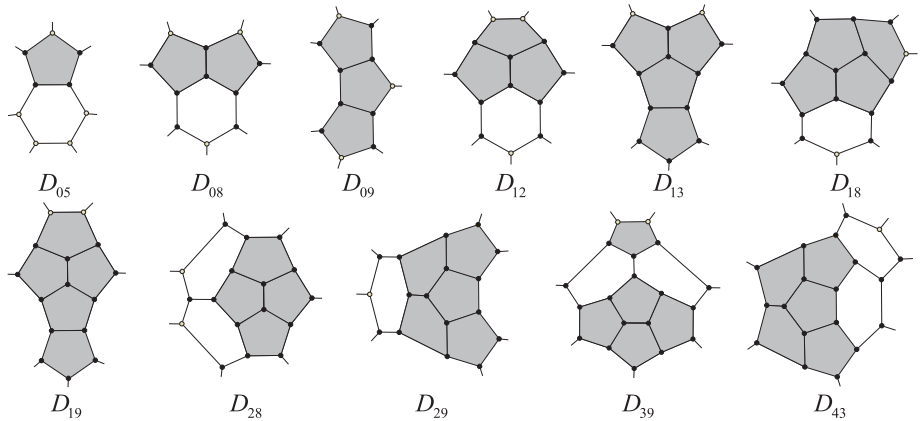}
\caption{\label{patches}{\small Eleven patches of fullerenes.}}
\end{figure}
In the following, we set $\mathcal{D}:=\{D_{05}, D_{08}, D_{09}, D_{12}, D_{13}, D_{18}, D_{19}, D_{28}, D_{29}, D_{39}, D_{43}\}$.

\begin{lem}[\cite{cyclic-7}]\label{formula}
Let $C$ be an edge-cut in a fullerene graph $F$ and $H$ a component of $F-C$.
Let $n_1$ and $n_2$ be the numbers of vertices of degree one and two, $f_5$ the number of
pentagons, and $l$ the size of the outer face of $H$. Then, $6-f_5=4n_1+2n_2-l$.
\end{lem}
Recall that a \emph{bridge} of a graph $G$ is an edge $e$ such that $G-e$ has more connected components than $G$.
\begin{thm}[\cite{AK}]\label{unique-pm}
Let $G$ be a graph with a unique perfect matching. Then $G$
has a bridge belonging to the perfect matching.
\end{thm}
\begin{prop}\label{cyclic-6-edge-cut}
Suppose that $H$ is an induced subgraph of a fullerene $F$ and $H$ has a unique perfect matching $M$. If $\triangledown_F(H)$ is a cyclic $6$-edge-cut, then $H\cong J_1$ \emph{(}see Fig. \emph{\ref{sub-patches}}\emph{)}. If $\triangledown_F(H)$ is an $8$-edge-cut and $H$ has not $1$-degree vertex, then $H\cong J_2$ \emph{(}see Fig. \emph{\ref{sub-patches}}\emph{)}.
\end{prop}
\begin{proof}
Since $H$ has a unique perfect matching $M$, $H$ has a bridge $e$ with $e\in M$ by Theorem \ref{unique-pm}.
Clearly, deleting $e$ from $H$ makes two new connected components. We denote one of them by $H_1$, and set $H_2:=H-H_1$.
$|\triangledown_F(H_i)|$ is odd and $|\triangledown_F(H_i)|\geq 3$, $i=1, 2$ since $H_i$ has odd number of vertices and $F$ is a $3$-connected cubic graph. Without loss of generality, we suppose that $|\triangledown_F(H_1)|\leq|\triangledown_F(H_2)|$.

If $\triangledown_F(H)$ is a cyclic $6$-edge-cut of $F$, then $|\triangledown_F(H_1)|+|\triangledown_F(H_2)|=8$. So $\triangledown_F(H_1)$ is a $3$-edge-cut and $\triangledown_F(H_2)$ is a $5$-edge-cut. By Lemma \ref{3-4-edge-cut},  $H_1$ is an isolated vertex. So $H_2$ has a cycle, and further $\triangledown_F(H_2)$ is a cyclic $5$-edge-cut. By Lemma \ref{component-of-non-trivial-cyclic-5-edge-cut}, $\triangledown_F(H_2)$ is a trivial cyclic $5$-edge-cut, otherwise, $H$ has at least two perfect matchings, a contradiction.
Since $\triangledown_F(H)$ is a cyclic $6$-edge-cut, $F-H$ is connected and has cycles. So $F-H$ has at least six vertices.
It follows that $H_2$ is a $5$-cycle. Then $H\cong J_1$.

If $\triangledown_F(H)$ is an $8$-edge-cut of $F$, then $|\triangledown_F(H_1)|+|\triangledown_F(H_2)|=10$.
Since $H$ has not $1$-degree vertex, both $H_1$ and $H_2$ have cycles. So $\triangledown_F(H_i)$ is a cyclic edge-cut of $F$, $i=1, 2$. By Lemma \ref{3-4-edge-cut}, $|\triangledown_F(H_i)|\geq 5$. So $\triangledown_F(H_i)$ is a cyclic $5$-edge-cut of $F$. Since $H$ has a unique perfect matching, $H_i$ is a $5$-cycle. So $H\cong J_2$.
\end{proof}

\section{Properties of fullerenes with $f(F)=3$}
We start with some notations concerning  a generalized patch $P$ of a fullerene introduced in
\cite{YangQ}. Label clockwise (counterclockwise) the half-edges of $P$ by $t_1, t_2, \ldots, t_k$, and set $a_i$ as the number of vertices from $t_i$ to $t_{i+1}$ in a clockwise (counterclockwise) scan of the boundary of $P$. Then the cyclic sequence $[a_1 a_2\ldots a_k]$ is called \emph{a distance-array} of $P$. Since a fullerene graph has only pentagonal and hexagonal faces, $1\leq a_i\leq6$.
For instance, a distance-array to describe the boundary of $J_1$ (see Fig. \ref{sub-patches}) could be $[132223]$.
Since we might start reading the boundary from different position and  in clockwise or counterclockwise direction, the boundary of a generalized patch may has more than one distance-arrays to describe it. However, we  easily   see that for the same boundary, the distance-arrays by  rotations and reversions are regarded as  equivalent.
\begin{figure}[htbp!]
\centering
\includegraphics[height=6.6cm]{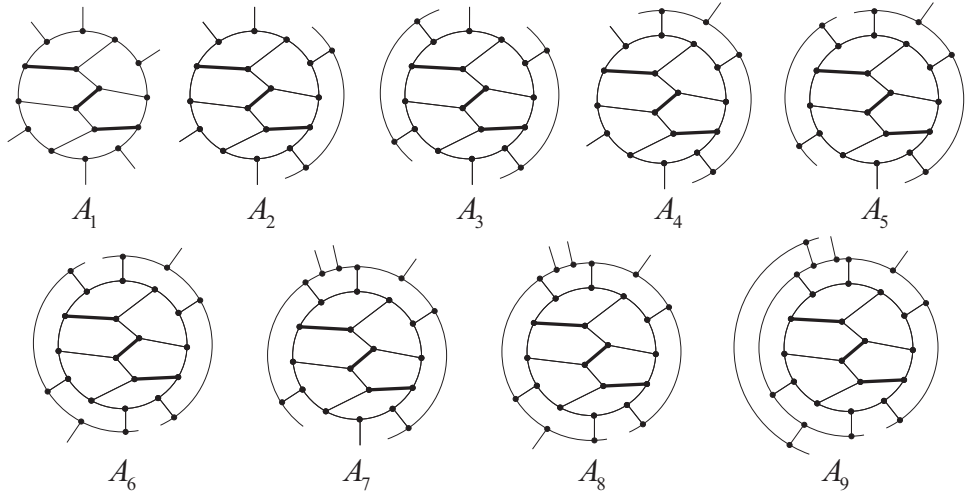}
\caption{\label{nine-caps}{\small Nine caps.}}
\end{figure}
We note that if $P$ has a distance-array $[a_1a_2\ldots a_k]$, it has at most $2k$ distinct distance-arrays describing it. Clearly, the smallest one in the numerical is uniquely determined.
In the following, we call such distance-array the \emph{min-distance-array} of $P$.
For example, for $P$ with a distance-array $[3516]$, there are eight distinct distance-arrays $[3516], [5163], [1635], [6351], [3615], [5361], [1536], [6153]$ to describe it, where $[1536]$ is the min-distance-array of $P$.

For nine caps in  Fig. \ref{nine-caps},  $A_1, A_3$ and $A_6$ have the  min-distance-array  $[234234]$,  $A_2, A_5$ and $A_8$ have $[233424]$, and $A_4, A_7$ and $A_9$ have $[233343]$. A nanotube fullerene of \emph{type $(4, 2)$} is comprised of two caps in Fig. \ref{nine-caps} with the same min-distance-array, and some layers of hexagons between them. The graph shown in Fig. \ref{example-type(4,2)} is a nanotube fullerene of type $(4, 2)$ with caps $A_1$ and $A_3$.
\begin{figure}[htbp!]
\centering
\includegraphics[height=4.2cm]{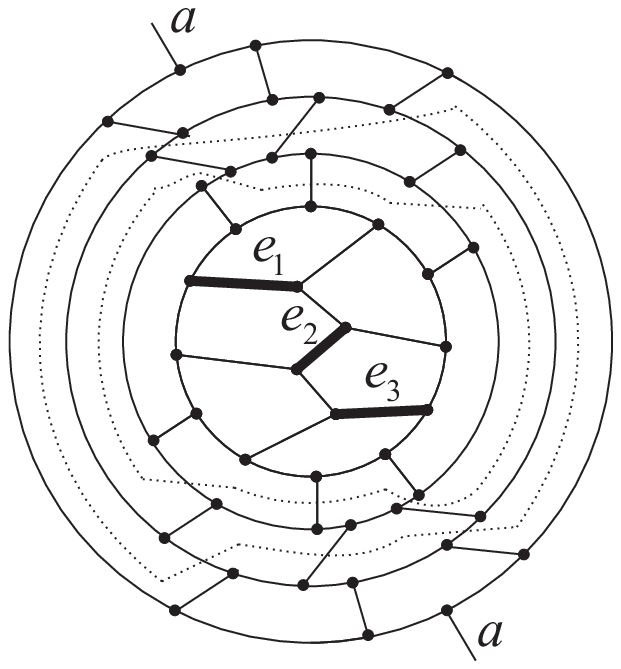}
\caption{\label{example-type(4,2)}{\small A nanotube fullerene of type $(4, 2)$ (each dashed cycle traverses a layer of hexagons).}}
\end{figure}

By Theorem $4$ and Corollary $5$ in \cite{FR}, the following lemma holds.
\begin{lem}\label{determine-nanotube-type(4,2)}
Let $F$ be a fullerene graph with a non-degenerate cyclic $6$-edge-cut $C$. If
the configurations of the six pentagons in one component of $F-C$ are as depicted in Fig. \emph{\ref{sub-patches}} $P^3$, then $F$ is a nanotube of type $(4, 2)$.
\end{lem}
\begin{lem}\label{(4,2)-nanotube}
Let $F$ be a nanotube fullerene of type $(4, 2)$. Then $f(F)=3$ and $F$ has a forcing set $S$ of size $3$ such that $F[V(S)]$ is isomorphic to $P_6$ with the incident edge set as depicted in $L_3$ \emph{(}see Fig. \emph{\ref{seed-case1}}\emph{)}.
\end{lem}
\begin{proof}
By the definition of a nanotube fullerene of type $(4, 2)$, $F$ has a cap $A_i$, $i\in\{1, \ldots, 9\}$ (see Fig. \ref{nine-caps}). Let $S$ be the set of the three dark edges in $A_i$ (see Fig. \ref{nine-caps}). Then we can check that $S$ is a forcing set of $F$. For example, $S:=\{e_1, e_2, e_3\}$ is a forcing set of the fullerene $F$ as shown in Fig. \ref{example-type(4,2)}.
\end{proof}

Let $G$ be a connected graph and $e$ a cut edge of $G$. The edge $e$ has one end $x$ in one component $G_1$ of $G-e$. If $G_1-x$ is empty or has a unique perfect matching, then we call $e$ a \emph{pendent edge} and $G_1$ a \emph{pendent blossom}, and say that $G_1$ is incident with $e$, and vice versa. The number of the adjacent vertices of $v$ in $G$ is called the \emph{degree} of $v$, denoted by $d_{G}(v)$.

Let $F$ be a fullerene and $S=\{e_1, e_2, e_3\}$ a forcing set of $F$. We define the following notations:

$F_0':= F[V(S)]$, $F_0'':=F- F_0'$;

If $F_i''$ has a $1$-degree vertex, then $F_{i+1}''$ is obtained from $F_i''$ by deleting this $1$-degree vertex and its adjacent vertex;

$F_{i+1}':=F- F_{i+1}''$;

$X_i$ denotes the set of edges in $F$ from $F_i'$ to $F_i''$.

Suppose that $F_k''$ ($k\geq0$) is the first such subgraph that has no $1$-degree vertices. In the following, we let $F'':=F_k''$ and $F':=F-F''$.
\begin{lem}\label{|X_i|}
For integer $k\geq1$, $|X_{i+1}|\leq |X_i|\leq 12$, $i=0, 1, \ldots, k-1$.
\end{lem}
\begin{proof}
Since $F_0'=F[V(S)]$ and $S$ is a set of three independent edges, $|X_0|\leq 12$.

If $F_{i+1}''$ is obtained from $F_i''$ by deleting a $1$-degree vertex and its adjacent vertex, then $|X_{i+1}|=|X_i|-2 +\delta_e$. Since $F$ is a cubic graph and $|F_i''|$ and $|F_{i+1}''|$ are even, $\delta_e=-2, 0$, or $2$. So $|X_{i+1}|\leq |X_i|$.
\end{proof}

For two subgraphs $A$ and $B$ of a graph $G$, we denote by $E(A, B)$ the set of edges of
$G$ with one end in $A$ and the other end in $B$, and set $e(A, B)=|E(A, B)|$.

\begin{thm}\label{Operation3}
If $F''=\emptyset$, then $F=F'$. If $F''\neq\emptyset$, then $|\triangledown_F(F'')|=8, 10,$ or $12$. Moreover, $F''\cong J_2$ \emph{(}see Fig. \emph{\ref{sub-patches}}\emph{)} if $|\triangledown_F(F'')|=8$.
Suppose that $F$ is not a nanotube fullerene of type $(4, 2)$. Then
$F''\cong PP$ \emph{(}see Fig. \emph{\ref{sub-patches})} or consists of a pentagon and a patch $H\in\mathcal{D}$ connecting by an edge if $|\triangledown_F(F'')|=10$,
and $F''$ consists of two patches $H_1\in\mathcal{D}$ and $H_2\in\mathcal{D}$ connecting by an edge if $F''$ has no pendent pentagons and $|\triangledown_F(F'')|=12$.
\end{thm}
\begin{proof}
Recall that $S=\{e_1, e_2, e_3\}$ is a forcing set of $F$.
We denote by $M$ the unique perfect matching of $F$ with $S\subseteq M$.
We suppose that $F''\neq\emptyset$.
Clearly, $F''$ is an induced plane subgraph of $F$ and it can be extended to $F$ on the plane.
By Lemma \ref{|X_i|}, $|\triangledown_F(F'')|\leq 12$. Since $F$ is $3$-regular and the order of $F''$ is even, $|\triangledown_F(F'')|$ is even.
Since $F''$ has no $1$-degree vertices and has a unique perfect matching, $F''$ is connected by Lemma \ref{3-4-edge-cut} and Proposition \ref{cyclic-6-edge-cut}.
Since $F''$ has a unique perfect matching $M\cap E(F'')$, by Theorem \ref{unique-pm}, $F''$ has a cut-edge $e''=v_1v_2\in M\cap E(F'')$ which separates $F''$ into two components $H_1$ and $H_2$ with $v_i\in V(H_i)$, $i=1, 2$.
Clearly, the order of $H_i$ is odd. So $|\triangledown_F(H_i)|$ is odd.
Since $F''$ has not $1$-degree vertices, both $\triangledown_F(H_1)=E(H_1, F')\cup \{e''\}$ and $\triangledown_F(H_2)=E(H_2, F')\cup \{e''\}$ are cyclic edge-cut of $F$.
Since $c\lambda(F)=5$, we suppose that  $|\triangledown_F(H_1)|\geq|\triangledown_F(H_2)|\geq5$.
So $|\triangledown_F(F'')|=|\triangledown_F(H_1)|+|\triangledown_F(H_2)|-2\geq8$.
Hence $|\triangledown_F(F'')|=8, 10,$ or $12$.
If $|\triangledown_F(F'')|=8$, then $F''\cong J_2$ by Proposition \ref{cyclic-6-edge-cut}. Next, we suppose that $|\triangledown_F(F'')|=10$ or $12$.

\emph{Case $1$.} $|\triangledown_F(F'')|=10$.

Since $|\triangledown_F(H_1)|+|\triangledown_F(H_2)|=|\triangledown_F(F'')|+2=12$, $|\triangledown_F(H_2)|=5$ and $|\triangledown_F(H_1)|=7$.
By Lemma \ref{component-of-non-trivial-cyclic-5-edge-cut}, $\triangledown_F(H_2)$ is a trivial cyclic 5-edge-cut of $F$ since $H_2-v_2$ has a unique perfect matching. So $H_2$ is isomorphic to a $5$-cycle.
Clearly, $H_1$ has not $1$-degree vertex or has a unique $1$-degree vertex $v_1$. For case $d_{H_1}(v_1)=1$, $\triangledown_F(H_1-v_1)$ is a cyclic $6$-edge-cut of $F$. Since $H_1-v_1$ has a unique perfect matching, $H_1-v_1\cong J_1$ by Proposition \ref{cyclic-6-edge-cut}. So $F''\cong PP$. If $d_{H_1}(v_1)=2$, then $H_1$ is $2$-connected, otherwise, $F$ has a cyclic edge-cut of size at most four, a contradiction. Moreover, any inner vertex of $H_1$ has degree $3$. So $H_1$ is a patch of $F$, and the boundary of $H_1$ is a cycle, denoted by $C$.
We claim that $H_1$ has at most five pentagons (then $H_1\in\mathcal{D}$ by Lemma \ref{chooseD}). By the contrary, we suppose that $H_1$ has at least six pentagons.

Since $|\triangledown_F(H_1)|=7$, by Lemma \ref{formula}, the length of $C$ is at least $14$.
Let $C=v_1x_1x_2\cdots$\\
$x_{k-1}x_{k}v_1$, $k\geq13$. We note that there are exactly seven $2$-degree vertices in $H_1$, all of which belong to $C$. Clearly, $\triangledown_F(H_1-v_1)$ is an $8$-edge-cut of $F$. If $H_1-v_1$ has not $1$-degree vertex, then $H_1-v_1\cong J_2$ by Proposition \ref{cyclic-6-edge-cut} since $H_1-v_1$ has a unique perfect matching. This contradicts that $H_1-v_1$ has at least five pentagons. So $H_1-v_1$ has a $1$-degree vertex. Since $d_{H_1}(v_1)=2$ and $H_1$ is $2$-connected, any $1$-degree vertex of $H_1-v_1$ is adjacent to $v_1$ in $H_1$. So only $x_1$ and $x_{k}$ may be $1$-degree vertex in $H_1-v_1$. We consider the following two cases.

\emph{Subcase $1.1.$} $d_{H_1-v_1}(x_1)=1$ and $d_{H_1-v_1}(x_{k})\neq1$.

In this case, $d_{H_1}(x_1)=2$ and $d_{H_1}(x_{k})=3$. So $x_1x_2\in M$. Clearly, $H_1^1:=H_1-\{v_1, x_1, x_2\}$ has a unique perfect matching, and has at least 5 pentagons if $d_{H_1}(x_2)=2$, at least 4 pentagons if $d_{H_1}(x_2)=3$.
If $d_{H_1}(x_2)=2$, then $\triangledown_F(H_1^1)$ is a cyclic $6$-edge-cut of $F$. By Proposition \ref{cyclic-6-edge-cut}, $H_1^1\cong J_1$, a contradiction. Hence $d_{H_1}(x_2)=3$ and $\triangledown_F(H_1^1)$ is an $8$-edge-cut of $F$. By  Proposition \ref{cyclic-6-edge-cut}, $H_1^1$ has a $1$-degree vertex. We can check that any inner vertex of $H_1$ can not be $1$-degree vertex in $H_1^1$, and $d_{H_1^1}(x_{k})\neq1$. So $d_{H_1^1}(x_{3})=1$. That is, $d_{H_1}(x_{3})=2$ and $x_3x_4\in M$. For $H_1^2:=H_1-\{v_1, x_1, x_2, x_3, x_4\}$, we can similarly show that $d_{H_1}(x_{4})=3$, $d_{H_1}(x_{5})=2$ and $x_5x_6\in M$. Set $H_1^3:=H_1-\{v_1, x_1, \ldots, x_6\}$. $H_1^3$ has a unique perfect matching, and has at least 3 pentagons if $d_{H_1}(x_6)=2$, at least 2 pentagons if $d_{H_1}(x_6)=3$. As the above discussion, $d_{H_1}(x_6)=3$. If $H_1^3$ has no 1-degree vertex, then $H_1^3\cong J_2$ by Proposition \ref{cyclic-6-edge-cut}. So $H_1\cong G_1$ (see Fig. \ref{Claim2-caes1} ($a$)).
\begin{figure}[htbp!]
\centering
\includegraphics[height=4.0cm]{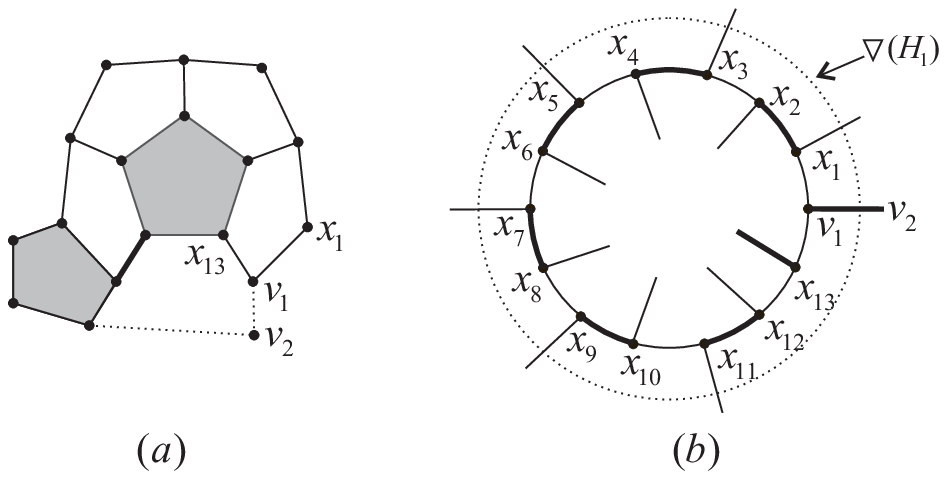}
\caption{\label{Claim2-caes1}{\small (a) A patch $G_1$ (not include $v_2$); (b) The edge set traversed by dashed line is $\triangledown_F(H_1)$.}}
\end{figure}
Since the length of any facial-cycle of $F$ is five or six, $v_2$ is adjacent to two vertices in $H_1$, this contradicts that there is exactly one edge connecting $H_1$ and $H_2$. So $H_1^3$ has a $1$-degree vertex. We can check that only $x_7$ can be a $1$-degree vertex in $H_1^3$. Hence $d_{H_1}(x_7)=2$ and $x_7x_8\in M$.
For $H_1^4:=H_1-\{v_1, x_1, \ldots, x_8\}$, as the discussion of $H_1^3$, we have $d_{H_1}(x_8)=3, d_{H_1}(x_9)=2$ and $x_9x_{10}\in M$. Since $|\triangledown_F(H_1)|=7$, $e(H_1, H_2)=1$ and any facial cycle of $F$ is of length $5$ or $6$, $d_{H_1}(x_{10})=3$.
Set $H_1^5:=H_1-\{v_1, x_1, \ldots, x_{10}\}$. So $\triangledown_F(H_1^5)$ is an $8$-edge-cut of $F$ and $H_1^5$ has a $1$-degree vertex. We can check that $d_{H_1^5}(x_{11})=1$. Hence the incident vertex set of $\triangledown_F(H_1)$ in $H_1$ is $\{v_1, x_1, x_3, x_5, x_7, x_9, x_{11}\}$ (see Fig. \ref{Claim2-caes1} (b)).
So $d_{H_1}(x_{i})=3$, $i=12, \ldots, k$.
Since $H_2$ is a pentagon and the length of any facial cycle of $F$ is five or six, $k=13$ and $x_{11}$ has an adjacent vertex in $H_2$. So $e(H_1, H_2)\geq2$. This contradicts that there is only one edge between $H_1$ and $H_2$.
So Subcase $1.1$ can not happen.

\emph{Subcase $1.2.$} $d_{H_1-v_1}(x_1)=1$ and $d_{H_1-v_1}(x_{k})=1$.
\begin{figure}[htbp!]
\centering
\includegraphics[height=4.0cm]{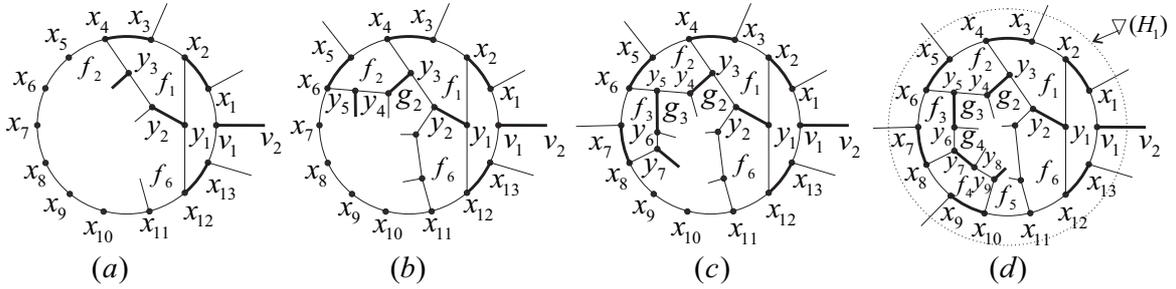}
\caption{\label{Case2}{\small Illustration for the proof of Subcase $1.2$ (I).}}
\end{figure}

Then $d_{H_1}(x_1)=2$ and $d_{H_1}(x_{k})=2$, $x_1x_2\in M$ and $x_{k}x_{k-1}\in M$.
Let $H_{1,1}=H_1-\{v_1, x_1, x_{k}\}$ and $\bar{H}_{1,1}=F-H_{1,1}$. We note that $\triangledown_F(H_{1,1})$ is a cyclic $6$-edge-cut of $F$ and $|\bar{H}_{1,1}|\geq|\{v_1, x_1, x_k\}\cup V(H_2)\cup V(S)|=3+5+2\times3=14$.

\emph{Claim 1:} $\triangledown_F(H_{1,1})$ is a non-degenerate cyclic $6$-edge-cut of $F$.

By the contrary, we suppose that $\triangledown_F(H_{1,1})$ is degenerate. So $H_{1,1}$ or $\bar{H}_{1,1}$ has at most five pentagons.
If $H_{1,1}$ has at most five pentagons, then it has exactly five pentagons since $H_1$ has at least six pentagons.
So $x_2$ and $x_{k-1}$ are adjacent in $H_1$. By Lemma \ref{5-pentagons-cylic-6-edge-cut}, $H_{1,1}\cong P^1$ or $P^2$.
We notice that vertices $x_2$ and $x_{k-1}$ are two $2$-degree adjacent vertices in $H_{1,1}$.
It is easy to check that $H_{1,1}-\{x_2, x_{k-1}\}$ is $2$-connected and has at least two perfect matchings.
So $H_1-v_1$ has at least two perfect matchings, a contradiction.
If $\bar{H}_{1,1}$ has at most five pentagons, by Lemma \ref{5-pentagons-cylic-6-edge-cut}, $\bar{H}_{1,1}\cong P^1$ or $P^2$ since $|\bar{H}_{1,1}|\geq14$. We notice that both $x_1$ and $x_k$ are $2$-degree vertices in $\bar{H}_{1,1}$ which are connected by a $2$-path $x_1v_1x_k$ on the boundary of $\bar{H}_{1,1}$. If $\bar{H}_{1,1}\cong P^2$, then for any case of $\{x_1, x_k\}$, $\bar{H}_{1,1}-(\{x_1, v_1, x_k\}\cup V(H_2))$ is a path $P_8$ of order $8$. Clearly, $S\subset E(P_8)$ can not force a perfect matching of $P_8$ in the initiative graph $F$, a contradiction.
For the case $\bar{H}_{1,1}\cong P^1$, we can show that $\bar{H}_{1,1}-(\{x_1, v_1, x_k\}\cup V(H_2))$ has not perfect matching, a contradiction.
Hence $\triangledown_F(H_{1,1})$ is a non-degenerate cyclic $6$-edge-cut of $F$.

By Lemma \ref{formula}, the boundary $C$ of $H_1$ is a $14$-cycle, $k=13$. By Claim 1, $x_2x_{12}\notin E(F)$.
Further, $d_{H_1}(x_2)=3$ and $d_{H_1}(x_{12})=3$, otherwise, $\triangledown_F(H_1-\{v_1, x_1, x_2, x_{13}\})$ or $\triangledown_F(H_1-\{v_1, x_1, x_{13}, x_{12}\})$ is a non-degenerate cyclic $5$-edge-cut of $F$, then $H_1-v_1$ has at least two perfect matchings, a contradiction.
So there is a vertex $y_1\in V(H_{1,1})$ lying in the interior area of $C$  such that $x_2y_1$, $x_{12}y_1\in E(H_{1,1})$.
Since the cyclic edge-connectivity of $F$ is five and each $5$-cycle (resp. $6$-cycle) is a facial cycle, $y_1$ is not adjacent to $x_3, x_4, \ldots, x_{11}$. So there is $y_2\in V(H_1)\setminus\{v_1, x_1, \ldots, x_{13}\}$ such that $y_1y_2\in M$.
We note that $H_{1, 1}$ is a patch of $F$ with $\triangledown_F(H_{1, 1})$ being a non-degenerate cyclic $6$-edge-cut, and $x_2y_1x_{12}$ is a $2$-path on the boundary of $H_{1, 1}$, $H_{1, 1}-\{x_2, x_{12}\}$ has a unique perfect matching $M\cap E(H_{1, 1}-\{x_2, x_{12}\})$.

\emph{Claim 2:} Exactly one of $x_3$ and $x_{11}$ is $2$-degree vertex in $H_1$. Suppose that $d_{H_1}(x_3)=2$, then $d_{H_1}(x_{4})=3$ and $y_1$, $x_2, x_3, x_4$ belong to a hexagonal face $f_1$ of $H_1$ (see Fig. \ref{Case2} (a)).

Clearly, $H_{1, 2}:=H_1-\{v_1, x_1, x_2, x_{12}, x_{13}, y_1, y_2\}$  has a unique perfect matching.
Since $\triangledown_F(H_{1, 2})$ is an $8$-edge-cut of $F$ and $H_{1, 2}$ has at least three pentagons, $H_{1, 2}$ has a $1$-degree vertex by Proposition \ref{cyclic-6-edge-cut}. We can check that only $x_3$ and $x_{11}$ may be $1$-degree in $H_{1, 2}$.

If both $x_3$ and $x_{11}$ are $1$-degree in $H_{1, 2}$, then
$x_3x_4, x_{11}x_{10}\in M$. Let $H_{1,1}':=H_1-\{v_1, x_1, x_2, x_3, x_{11}, x_{12}, x_{13}, y_1\}$. Since
$\triangledown_F(H_{1,1}')$ is a non-trivial cyclic $5$-edge-cut of $F$,  by Lemma \ref{component-of-non-trivial-cyclic-5-edge-cut}, $H_{1,1}'$ is a patch that consists of a pentacap and $i\geq0$ layers of hexagons around it. We notice that $x_4, y_2, x_{10}$ are three consecutive $2$-degree vertices on the boundary of $H_{1,1}'$. It is easy to check that $H_{1,1}'-\{x_4, y_2, x_{10}\}$ has at least two perfect matchings. So $H_1-v_1$ has at least two perfect matchings, a contradiction.
So exactly one of $x_3$ and $x_{11}$ is  $2$-degree in $H_1$.

We suppose that $d_{H_1}(x_{11})=3$, $d_{H_1}(x_3)=2$.  Then $x_3x_4\in M$.
Let $H_{1,1}'':=H_1-\{v_1, x_1, x_2, x_3, x_4, x_{12}, x_{13}, y_1\}$.
If $d_{H_1}(x_4)=2$, then $\triangledown_F(H_{1,1}'')$ is a non-degenerate cyclic $5$-edge-cut. So $H_{1,1}''-y_2$ has at least two perfect matchings. It follows that $H_1-v_1$ has at least two perfect matchings, a contradiction. So $d_{H_1}(x_4)=3$.

Clearly $H_{1, 2}':=H_1-\{v_1, x_1, x_2, x_3, x_4, x_{12}, x_{13}, y_1, y_2\}$ has a unique perfect matching $M\cap E(H_{1, 2}')$. The two vertices $y_2$ and $x_4$ are not adjacent in $F$, otherwise, $\triangledown_F(H_{1, 2}')$ is a cyclic $6$-edge-cut of $F$, then $H_{1, 2}\cong J_1$ by Proposition \ref{cyclic-6-edge-cut}, this contradicts that $H_{1, 2}'$ has at least three pentagons.
So $f_1$ is a hexagon (see Fig. \ref{Case2}($a$)). Then the Claim 2 is done.

Let $H_{1, 3}:=H_1-\{v_1, x_1, x_2, x_3, x_{12}, x_{13}, y_1\}$. We notice that $\triangledown_F(H_{1, 3})$ is a cyclic $6$-edge-cut of $F$ and $M\cap E(H_{1, 3})$ is not a perfect matching of $H_{1, 3}$. We can check that $H_{1, 3}$ has not $1$-degree vertex. By Lemma \ref{3-4-edge-cut}, $H_{1, 3}$ is $2$-connected. So $H_{1, 3}$ is a patch of $F$.

If $f_6$ (see Fig. \ref{Case2}($a$)) is a pentagon, then $\triangledown_F(H_{1, 3})$ is degenerate and $H_{1, 3}$ has exactly five pentagons. So $H_{1, 3}\cong P^1$ or $P^2$ by Lemma \ref{5-pentagons-cylic-6-edge-cut}.
We notice that $x_4, y_2, x_{11}$ are three $2$-degree vertices in $H_{1, 3}$, and $x_4, y_2$ are connected by a $2$-path, $y_2, x_{11}$ are connected by a $2$-path on the boundary of $H_{1, 3}$.
Since $H_{1, 3}-\{x_4, y_2\}$ has a unique perfect matching $M\cap E(H_{1, 3}-\{x_4, y_2\})$, $H_1\cong R_1$ or $R_2$ (see Fig. \ref{R_i}).
\begin{figure}[htbp!]
\centering
\includegraphics[height=4cm]{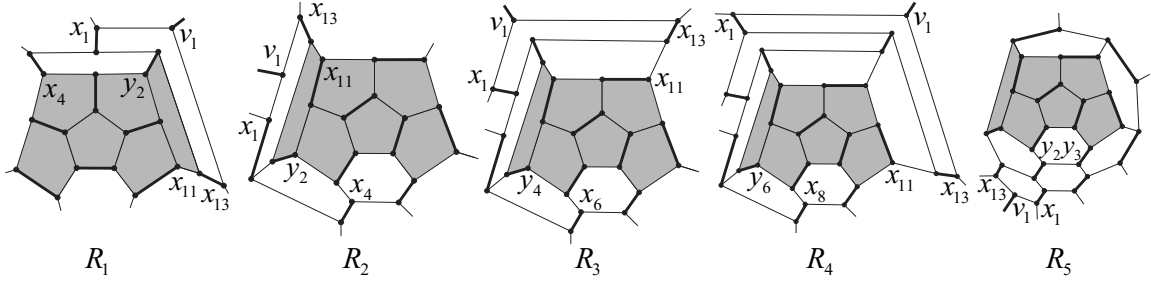}
\caption{\label{R_i}{\small Illustration for the proof of Subcase $1.2$ (II).}}
\end{figure}

If $f_6$ (see Fig. \ref{Case2}($a$)) is a hexagon, then $\triangledown_F(H_{1, 3})$ is non-degenerate.
We can check that $H_{1, 3}$ is a patch of $F$, $x_4y_3y_2$ is a $2$-path on the boundary of $H_{1, 3}$ and $H_{1, 3}-\{x_4, y_2\}$ has a unique perfect matching $M\cap E(H_{1, 3}-\{x_4, y_2\})$.
As the proof of Claim 2, we obtain that $d_{H_1}(x_5)=2$ and $d_{H_1}(x_6)=3$, $f_2$ is a hexagon (see Fig. \ref{Case2}($b$)). Let $H_{1, 4}:=H_1-\{v_1, x_1, x_2, x_3, x_4, x_5, x_{12}, x_{13}, y_1, y_2, y_3\}$. Clearly, $\triangledown_F(H_{1, 4})$ is a cyclic $6$-edge-cut of $F$. As the proof of $H_{1, 3}$, we know that $H_{1, 4}$ is a patch of $F$.
If $g_2$ is a pentagon, then $\triangledown_F(H_{1, 4})$ is degenerate and $H_{1, 4}$ has five pentagons. So $H_{1, 4}\cong P^1$ or $P^2$.
We note that the distance-array of $H_{1, 4}$ has three successive $3$ and the first $3$ corresponds to $x_6x_5$ and $y_4y_3$. Since $H_{1, 4}-\{x_6, y_4\}$ has a unique perfect matching, $H_{1, 4}\cong P^2$. So $H_1\cong R_3$.
If $g_2$ is a hexagon, then $\triangledown_F(H_{1, 4})$ is non-degenerate. We note that $x_6y_5y_4$ is a $2$-path on the boundary of $H_{1, 4}$ and $H_{1, 4}-\{x_6, y_4\}$ has a unique perfect matching $M\cap E(H_{1, 4}-\{x_6, y_4\})$.
As the proof of Claim 2, we have $d_{H_1}(x_7)=2$ and $d_{H_1}(x_8)=3$, $f_3$ is a hexagon (see Fig. \ref{Case2}($c$)).
We can check that $H_{1, 5}:=H_1-\{v_1, x_1, x_2, \ldots, x_7, x_{12}, x_{13}, y_1, y_2, \ldots, y_5\}$ is a patch of $F$ and $\triangledown_F(H_{1, 5})$ is a cyclic $6$-edge-cut. If $g_3$ is a pentagon, then
$\triangledown_F(H_{1, 5})$ is degenerate and $H_{1, 5}$ has five pentagons.
By Lemma \ref{5-pentagons-cylic-6-edge-cut}, $H_{1, 5}\cong P^2$ since the distance-array of $H_{1, 5}$ has four consecutive $3$. So $H_1\cong R_4$ since $H_{1, 5}-\{x_8, y_6\}$ has a unique perfect matching. If $g_3$ is a hexagon, then $\triangledown_F(H_{1, 5})$ is non-degenerate, as the above discussion, we know that $d_{H_1}(x_9)=2$, $d_{H_1}(x_{10})=3$ and $f_4$ is a hexagon (see Fig. \ref{Case2}(d)). Then the incident vertex set of the seven edges in $\triangledown_F(H_1)$ is $\{v_1, x_1, x_3, x_5, x_7, x_9, x_{13}\}$(see Fig. \ref{Case2}($d$)).
Set $H_{1,6}:=H_1-\{y_1\}\cup V(C)$.
If $f_5$ is a pentagon, then $\triangledown_F(H_{1, 6})$ is a degenerate cyclic $6$-edge-cut of $F$ such that $H_{1, 6}$ has five pentagons. So $H_{1,6}\cong P^2$ since the distance array of $H_{1, 6}$ has five consecutive $3$.
We notice that $y_2$ and $y_3$ are two adjacent $2$-degree vertices on the boundary of $H_{1, 6}$. So $H_1\cong R_5$ (see Fig. \ref{R_i}).
If $f_5$ is a hexagon, then $\triangledown_F(H_{1, 6})$ is a non-degenerate cyclic $6$-edge-cut of $F$ and has the same distance-array as the patch $H_{1, 1}$. In an inductive way, we can show that
the configurations of the six pentagons in the component $H_{1, 1}$ are as depicted in Fig. \ref{sub-patches} $P^3$.
By Lemma \ref{determine-nanotube-type(4,2)}, $F$ is a nanotube fullerene of type $(4, 2)$, a contradiction.
So Subcase 1.2 can not happen.

\emph{Case $2$.} $|\triangledown_F(F'')|=12$ and $F''$ has no pendent pentagons.

Since $F''$ has no pendent pentagons, both $\triangledown_F(H_1)$ and $\triangledown_F(H_2)$ are not cyclic $5$-edge-cut of $F$, otherwise, $H_1$ or $H_2$ is a pendent pentagon of $F''$ since $F''$ has a unique perfect matching.
So $|\triangledown_F(H_i)|\geq7$, $i=1, 2$. This implies that $e(H_1, F')\geq 6$ and $e(H_2, F')\geq 6$. Hence $e(H_1, F')+e(H_2, F')\geq 12$. Since $|\triangledown_F(F'')|=e(H_1, F')+e(H_2, F')=12$, both $\triangledown_F(H_1)$ and $\triangledown_F(H_2)$ are cyclic $7$-edge-cuts of $F$.

Since $F''$ is connected, $H_1$ is connected. We claim that $H_1$ has not $1$-degree vertex.
If not, $H_1$ has a $1$-degree vertex. We note that only $v_1$ can be a $1$-degree vertex in $H_1$ since $F''$ has not $1$-degree vertices. So $d_{H_1}(v_1)=1$ and $\triangledown_F(H_1-v_1)$ is a cyclic $6$-edge-cut of $F$. Since $H_1-v_1$ has a unique perfect matching $M\cap E(H_1-v_1)$, $H_1-v_1\cong J_1$ by Proposition \ref{cyclic-6-edge-cut}. So $F''$ has a pendant pentagon, a contradiction. It means that $H_1$ has not $1$-degree vertex. Since $c\lambda(F)=5$ and $|\triangledown_F(H_1)|=7$, $H_1$ is $2$-connected and any inner vertex of $H_1$ has degree $3$. So $H_1$ is a patch of $F$. Similarly, $H_2$ is also a patch of $F$.

Obviously, $F$ has exactly twelve pentagons. We claim that $H_i$ has at most five pentagons, $i=1, 2$.  If not, we suppose that $H_1$ has at least six pentagons. As the proof of Subcases 1.1 and 1.2, we can also obtain a contradiction.
So $H_1$ has at most five pentagons. Similarly, $H_2$ has at most five pentagons. Since both $H_1-v_1$ and $H_2-v_2$ have a unique perfect matching, $H_1, H_2\in\mathcal{D}$ by Lemma \ref{chooseD}.
\end{proof}

\begin{cor}\label{terminal-gene-patch}
In the condition of Theorem \emph{\ref{Operation3}},
$F''\cong PP$ $($see Fig. \emph{\ref{sub-patches}}$)$ or $F''\cong T_1, T_2, \ldots T_{16}$ $($see Fig. \emph{\ref{terminal10}}$)$ if $|\triangledown_F(F'')|=10$.
$F''\cong T_{17}, T_{18}, \ldots T_{180}$ $($see Fig. \emph{\ref{terminal12I}}, \emph{\ref{terminal12II}} and \emph{\ref{terminal12III}}$)$ if $|\triangledown_F(F'')|=12$ and $F''$ has no pendent pentagons.
\end{cor}

\begin{lem}\label{main-lem}
Let $H$ be an induced subgraph of fullerene $F$ with a unique perfect matching $M$ and having no $1$-degree vertices. Then for any isolated edge $e_3$ in $F-H$, we have

\emph{(}i\emph{)} If $H$ consists of two patches $H_x$ and $H_y$ connecting by an edge $xy\in M$ with $x\in H_x, y\in H_y$ and $|\triangledown_F(H_x)|=|\triangledown_F(H_y)|=7$, then $e(e_3, H_x)=e(e_3, H_y)=2$. Moreover, if $T:=F-H-V(e_3)$ is a generalized patch of $F$, then $F[V(H)\cup V(e_3)]$ is a patch as shown in Fig. \ref{full_structure1}, and $\triangledown_F(T)$ and $\triangledown_F(F[V(H)\cup V(e_3)])$ merge successively along the boundaries of $T$ and $F[V(H)\cup V(e_3)]$, respectively.
If we label the edges in $\triangledown_F(F[V(H)\cup V(e_3)])$ as shown in Fig. \ref{full_structure1}, then $T$ has the distance-array $[a_1a_2\cdots a_8]$ with $a_1=1$ and $a_{5}\in\{1, 2\}$.
\begin{figure}[htbp!]
\centering
\includegraphics[height=3.2cm]{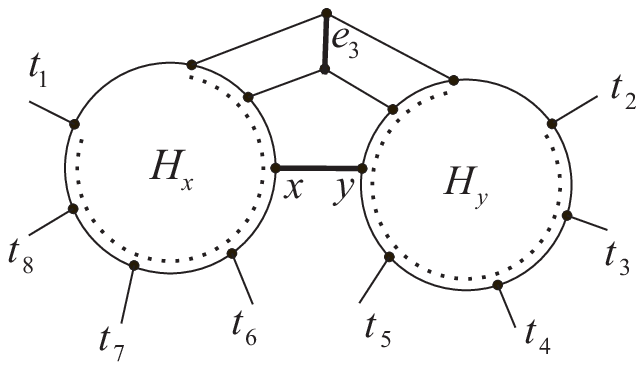}
\caption{\label{full_structure1}{\small $F[V(H)\cup V(e_3)]$.}}
\end{figure}

\emph{(}ii\emph{)} If $H$ has a pendent pentagon $P_x$ connecting by an edge $xy$ to $R:=H-P_x$ with $x\in P_x$ and $y\in R$, and $|\triangledown_F(R)|\leq9$, then $e(e_3, P_x)=1$. Moreover, if $T:=F-H-V(e_3)$ is a generalized patch of $F$ and $T$ has a perfect matching, then the three edges in $E(P_x, T)$ are incident with three successive vertices on $P_x$ respectively, and there are two elements $a_i$ and $a_{i+1}$ in the distance-array $[a_1a_2\cdots a_8]$ of $T$ such that $a_i, a_{i+1}\in\{3, 4\}$ \emph{(}here $a_9:=a_1$\emph{)}.
\end{lem}
\begin{proof}
From the assumption in (i), both $H_x$ and $H_y$ have cycles, and $|H_x|\geq7$, $|H_y|\geq7$. Since the cyclic edge-connectivity of $F$ is five, $e(e_3, H_x)\geq1$ and $e(e_3, H_y)\geq1$, otherwise, $\triangledown_F(H_y\cup e_3)$ or $\triangledown_F(H_x\cup e_3)$ is a cyclic $3$-edge-cut of $F$, a contradiction. If $e(e_3, H_x)=1$, then $\triangledown_F(H_y\cup e_3)$ is a cyclic $5$-edge-cut of $F$ that separates $H_x$ and $H_y$.
So $\triangledown_F(H_y\cup e_3)$ is a nontrivial cyclic $5$-edge-cut of $F$.
Let $A^i$ be a patch that consists of a pentacap and $i$ layers of hexagons (see Fig. \ref{main-lemfig} for $A^1$).
By Lemma \ref{component-of-non-trivial-cyclic-5-edge-cut}, $F[V(H_y)\cup V(e_3)]\cong A^i$ for some integer $i\geq0$, and $e_3$ is a boundary edge of $F[V(H_y)\cup V(e_3)]$.
Let $y'$ be a $2$-degree vertex in $A^i$.
For any boundary edge $e'$ of $A^i$ that is not incident with $y'$, $A^i-(\{y'\}\cup V(e'))$ has at least two perfect matchings or $A^i-V(e')$ has a $1$-degree vertex different from $y'$. So $H_y-y$ has at least two perfect matchings or $H_y$ has a $1$-degree vertex different from $y$. This contradicts that $H_y-y$ has a unique perfect matching and $H$ has not $1$-degree vertex. Hence $e(e_3, H_x)\geq2$. Similarly, $e(e_3, H_y)\geq2$. Since $e(e_3, H_x)+e(e_3, H_y)=4$,  $e(e_3, H_x)=e(e_3, H_y)=2$. Since any facial cycle of $F$ is a pentagon or hexagon and $E(H_x, H_y)=\{xy\}$, the two edges in $E(e_3, H_x)$ are incident with two ends of $e_3$.  Since $T$ is connected and $F$ is a plane graph, the two edges in $E(H_x, e_3)$ are incident with two successive $2$-degree vertices on the boundary of $H_x$ and the four edges in $E(H_x, T)$ are incident to four successive $2$-degree vertices on the boundary of $H_x$, respectively. For the six edges in $\triangledown_F(H_y)\setminus \{xy\}$, we have the same property. So $F[V(H)\cup V(e_3)]$ is a patch of $F$ as depicted in Fig. \ref{full_structure1}. Further, since $F$ is a plane graph, the edges in $\triangledown_F(T)$ and the edges in $\triangledown_F(H\cup V(e_3))$ merge successively along the boundaries of $T$ and $F[V(H)\cup V(e_3)]$. Since $E(H_x, H_y)=1$ and any facial cycle of $F$ is a pentagon or hexagon, $b_1=5$, $b_5\in\{4, 5\}$ in the distance-array $[b_1\cdots b_8]$ of $F[V(H)\cup V(e_3)]$.
So we have $a_1=1$, $a_{5}\in\{1, 2\}$ in the distance-array $[a_1a_2\cdots a_8]$ of $T$.
\begin{figure}[htbp!]
\centering
\includegraphics[height=4.0cm]{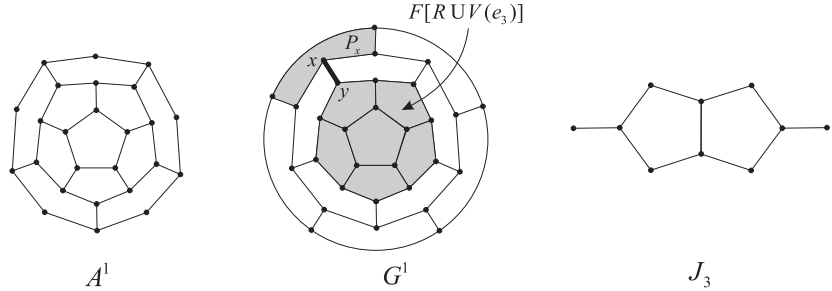}
\caption{\label{main-lemfig}{\small Illustration for the proof of Lemma \ref{main-lem}.}}
\end{figure}

For (ii), since $F''$ has no $1$-degree vertices, $R:=F''-P_x$ has a cycle. Since $P_x$ is a pentagon, $e(e_3, P_x)\leq1$. If $e(e_3, P_x)=0$, then $e(e_3, R)=4$ and $\triangledown_F(R\cup e_3)$ is a cyclic edge-cut of $F$ of size at most five since $|\triangledown_F(R)|\leq9$.
So $\triangledown_F(R\cup e_3)$ is a non-trivial cyclic $5$-edge-cut of $F$ with one component being $F[V(R)\cup V(e_3)]$.
By Lemma \ref{component-of-non-trivial-cyclic-5-edge-cut}, $F[V(R)\cup V(e_3)]\cong A^i$ for some integer $i\geq0$ ($A^i$ is as depicted in the above case).
If $i=0$, then $F\cong G^1$ (see Fig. \ref{main-lemfig}) since $E(R\cup e_3, P_x)=\{xy\}$ and $P_x$ is isomorphic to a $5$-cycle. This implies that $T\cong J_3$ as depicted in Fig. \ref{main-lemfig}. We note that $J_3$ has not perfect matching, a contradiction. So $i\geq1$.
Since $e(e_3, R)=4$, the degrees of the two ends of $e_3$ both are three in $F[V(R)\cup V(e_3)]$.
So $e_3$ is not on the boundary of $F[V(R)\cup V(e_3)]$. Let $y'$ be a $2$-degree vertex of $A^i$.
For each edge $e'$ not on the boundary of $A^i$, it is easy to check that $A^i-V(e')\cup\{y'\}$ has at least two perfect matchings.
So $R-y$ has at least two perfect matchings. This contradicts that $H$ has a unique perfect matching.
Hence $e(e_3, P_x)=1$. As the discussion of case $(i)$, if $T$ is a generalized patch of $F$, then the three edges in $E(P_x, T)$ are incident with three successive vertices on $P_x$, respectively. We suppose that the distance-array of $T$ is $[a_1\cdots a_8]$. Clearly, the three edges in $E(P_x, T)$ are three consecutive edges $t_i, t_{i+1}, t_{i+2}$ in $\triangledown_F(T)$. Since any facial cycle of $F$ is a pentagon or hexagon, $a_i, a_{i+1}\in\{3, 4\}$.
\end{proof}

\section{Generalized patches of fullerenes with $f(F)=3$}
A fullerene graph $F$ is called to be \emph{$1$-resonant} if for each hexagon $f$ of $F$, the subgraph $F-V(f)$ has a perfect matching.
\begin{thm}[\cite{1-resonant, LZ18}]\label{1-resonant}
Fullerene graph is $1$-resonant.
\end{thm}

\begin{thm}\label{initial_seed}
Let $F$ be a fullerene that is not nanotube of type $(4, 2)$.
If $S=\{e_1, e_2, e_3\}$ forces a perfect matching $M$ of $F$ and any proper subset of $S$ is not a forcing set of $F$, then $F\cong W_i$, $i\in\{1, \ldots, 81\}$ \emph{(}see Figs. \emph{\ref{fullerene1}} and \emph{\ref{fullerene2}}\emph{)}, or $F$ has a generalized patch $L\in \mathcal{L}$ such that $S\subseteq E(L)$ forces a perfect matching $M\cap E(L)$ of $L$ in the sense of $F$, where $\mathcal{L}:=\{L_1, L_2, \ldots, L_{95}\}$ \emph{(}see Figs. \emph{\ref{seed-case1}--\ref{seed-case3-3.2}}\emph{)}.
\end{thm}
\begin{proof}
By the assumption, $M$ is a unique perfect matching of $F$ with $S\subseteq M$. Set $F_0':=F[V(S)]$ and $F_0'':=F-F_0'$.
Obviously, $F_0'$ has three cases as follows.

\emph{Case 1.} $F_0'$ is connected.
\begin{figure}[htbp!]
\centering
\includegraphics[height=1.5cm]{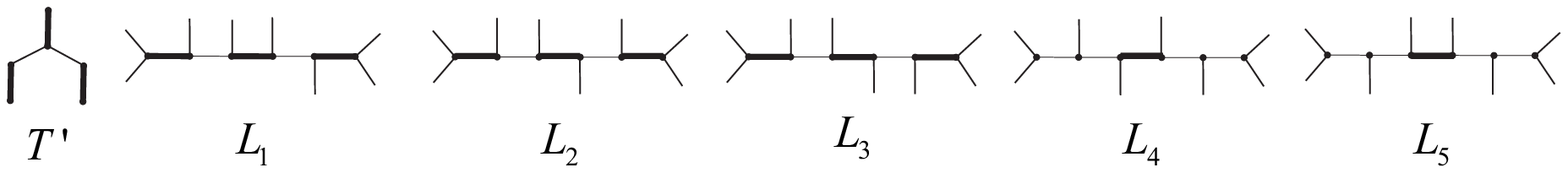}
\caption{\label{seed-case1}{\small $F_0'$ is connected.}}
\end{figure}

Since the order of $F_0'$ is six, $F_0'\cong P_6$, or $C_6$, or $J_1$ (see Fig. \ref{sub-patches}), or $T'$ (see Fig. \ref{seed-case1}). Since any proper subset of $S$ is not a forcing set of $F$, $F_0'\not\cong J_1$, $T'$. Clearly, each $6$-cycle of $F$ is a facial cycle. By Theorem \ref{1-resonant}, $F$ is $1$-resonant. Since $F-V(C_6)$ is $2$-connected, $F-V(C_6)$ has at least two perfect matchings. So $F_0'\cong P_6$. It follows that $F$ has a generalized patch $L_i$, $i\in\{1, \ldots, 5\}$ (see Fig. \ref{seed-case1}).
\begin{figure}[htbp!]
\centering
\includegraphics[height=3.9cm]{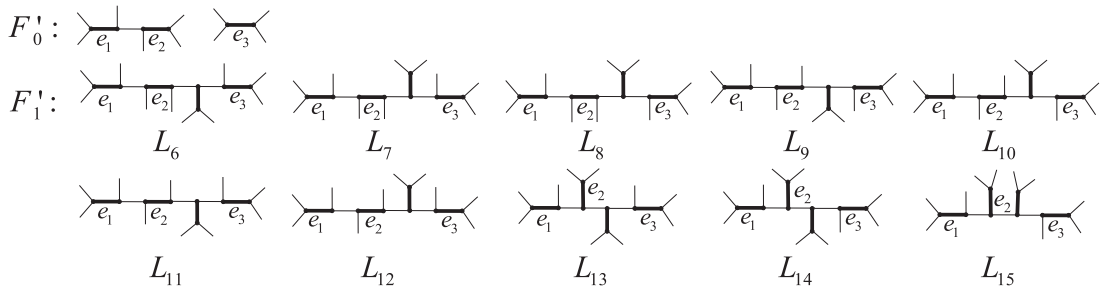}
\caption{\label{seed-case2-1}{\small $F_0'$ has two components (I).}}
\end{figure}

\emph{Case 2.} $F_0'$ has exactly two connected components, say $B_1$ and $B_2$.

Without loss of generality, we suppose that $|B_1|\geq|B_2|$. Then $B_1$ is a path with four vertices and $B_2$ is an edge. Clearly, $|\triangledown_F(F_0')|=10$. We notice that $F_0'':=F-F_0'$ has a unique perfect matching $M\setminus S$.

We claim that $F_0''$ has a $1$-degree vertex.
If not, then $F_0''$ has a pendent pentagon $P_x$ by Theorem \ref{Operation3}.
Obviously, $\triangledown_F(B_1)$ has exactly two cases as shown in Fig. \ref{seed-case2-1} and \ref{seed-case2-2}. So the distance-array of $B_1$ is $[132132]$ or $[122214]$. This is impossible by Lemma \ref{main-lem} $(ii)$. So $F_0''$ has a $1$-degree vertex.

For $F_0'$ in Fig. \ref{seed-case2-1},
the $1$-degree vertex of $F_0''$ has one adjacent vertex in $B_1$ and one in $B_2$ since any $5$-cycle of $F$ is a facial cycle and the length of the shortest cycle in $F$ is five. So $F$ has a generalized patch $L_i$,  $i\in\{6, \ldots, 15\}$.
\begin{figure}[htbp!]
\centering
\includegraphics[height=12.cm]{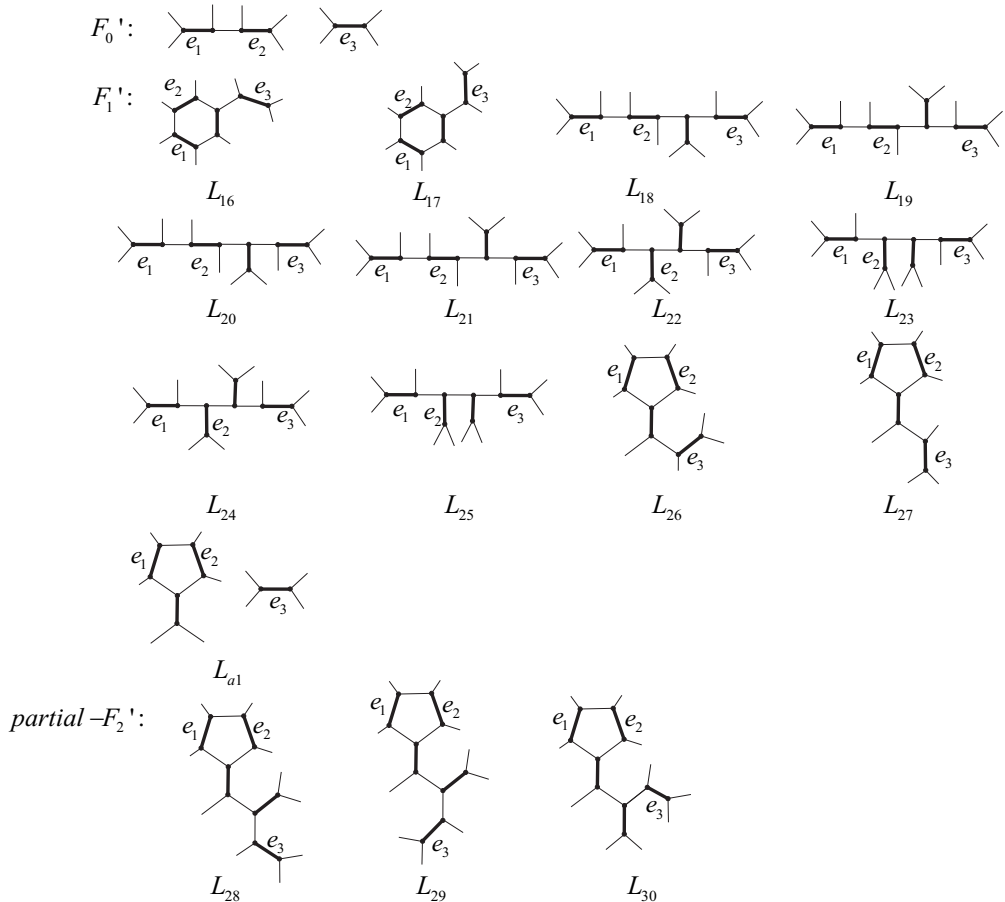}
\caption{\label{seed-case2-2}{\small $F_0'$ has two components (II).}}
\end{figure}

For $F_0'$ in Fig. \ref{seed-case2-2}, the $1$-degree vertex of $F_0''$ has one adjacent vertex in $B_1$ and one in $B_2$, or has two adjacent vertices in $B_1$. For the first case, $F$ has a generalized patch $L_i$,  $i\in\{16, \ldots, 25\}$. For the second case, $F$ has a subgraph $L_{a1}$ or a generalized patch $L_i$,  $i=26$ or $27$.
Since $|\triangledown_F(L_{a1})|=10$ and the distance-array of $L_{a1}-V(e_3)$ is $[132223]$, by Lemma \ref{main-lem} (ii), $F-V(L_{a1})$ has a $1$-degree vertex. Clearly, this $1$-degree vertex connects the two components of $L_{a1}$. So $F$ has a generalized patch $L_i$,  $i\in\{28, 29, 30\}$ or a generalized patch $L_j$, $j\in\{18, \ldots, 25\}$.

\emph{Case 3.} $F_0'$ consists of three independent edges $e_1$, $e_2$ and $e_3$ (see Fig. \ref{seed-case3}).

$F_0'':=F-F_0'$ has a unique perfect matching $M\setminus S$ and $|\triangledown_F(F_0'')|=12$.
If $F_0''$ has not $1$-degree vertices, then we claim that $F_0''$ has not pendent pentagons.
Suppose that $F_0''$ has a pendent pentagon $U$. Let edge $e\in M\setminus S$ be the pendent edge that is incident with $U$.
By Lemma \ref{main-lem} (ii), $e(U, e_i)=1$ for $i=1, 2, 3$. So $e(U, F_0''-U)=2$, a contradiction.
Hence $F_0''$ has not pendent pentagons. By Theorem \ref{Operation3}, $F_0''$ consists of two patches $U_0, V_0\in \mathcal{D}$ connecting by an edge $x_0\in M\setminus S$.
By Lemma \ref{main-lem} $(i)$, $e(U_0, e_i)=2=e(V_0, e_i)$ for $i=1, 2, 3$.
Since $F$ is a plane graph, the two edges in $E(U_0, e_i)$ are two edges in $\triangledown_F(U_0)$ (and $\triangledown_F(e_i)$) that are consecutive along the boundary of $U_0$ (and $e_i$) for $i=1, 2, 3$.
\begin{figure}[htbp!]
\centering
\includegraphics[height=3.6cm]{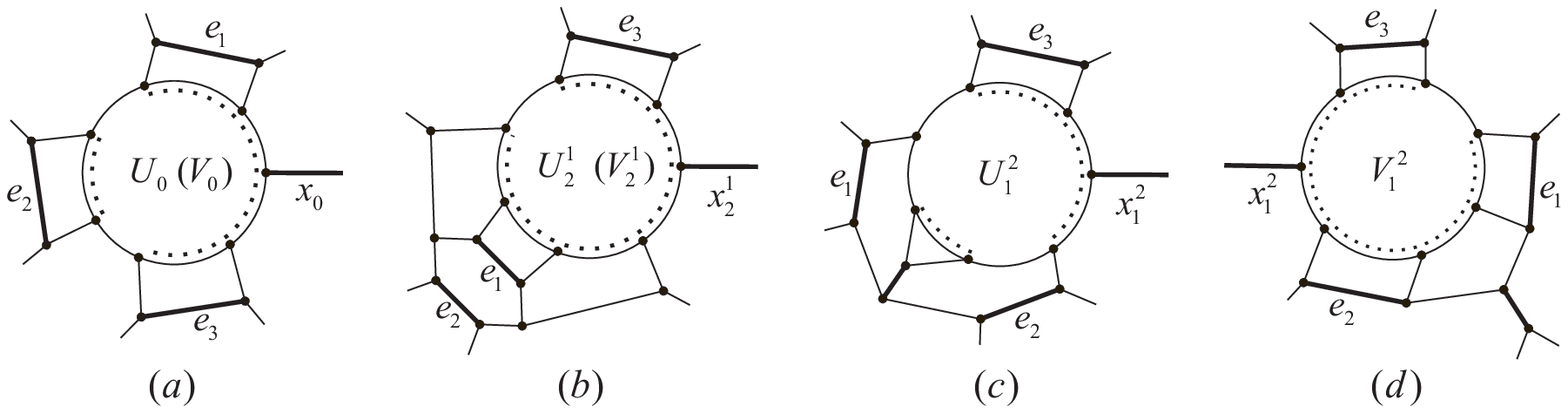}
\caption{\label{structure1}{\small }}
\end{figure}
Similarly, the two edges in $E(V_0, e_i)$ are two edges in $\triangledown_F(V_0)$ (and $\triangledown_F(e_i)$) that are consecutive along the boundary of $V_0$ (and $e_i$) for $i=1, 2, 3$.
Since $E(U_0, V_0)=\{x_0\}$ and any facial cycle of $F$ is $5$-cycle or $6$-cycle, the two edges in $E(U_0, e_i)$ are incident with two ends of $e_i$ respectively, $i=1, 2, 3$. So both $F[V(U_0)\cup V(S)]$ and $F[V(V_0)\cup V(S)]$ have the structures as depicted in Fig. \ref{structure1} $(a)$.
It implies that both $U_0$ and $V_0$ have the distance-array $[a_1\wedge a_2a_32a_52a_7]$, where $a_1, a_2=2$ or $3$, $a_3, a_5, a_7=3$ or $4$, and $\wedge$ marks the position from which the edge $x_0$ issues out.
We can check that the distance-arrays $[2\wedge232323]$ of $D_{08}$ and $[2\wedge243234]$ of $D_{29}$ satisfy the condition of $U_0$.
So $U_0=D_{08}$ or $D_{29}$. Similarly, $V_0=D_{08}$ or $D_{29}$.
Hence, $F\cong W_1, W_2,$ or $W_3$ (see Fig. \ref{fullerene1}).

If $F_0''$ has a $1$-degree vertex, then this vertex has one adjacent vertex in $V(e_1)$ and one in $V(e_2)$. So $F$ has a subgraph that is isomorphic to $L_{b1}$, $L_{b2}$, or $L_{b3}$, or has a generalized patch $L_i$, $i\in \{31, \ldots, 36\}$ (see Fig. \ref{seed-case3}).
Now we consider the following three cases.
\begin{figure}[htbp!]
\centering
\includegraphics[height=6.2cm]{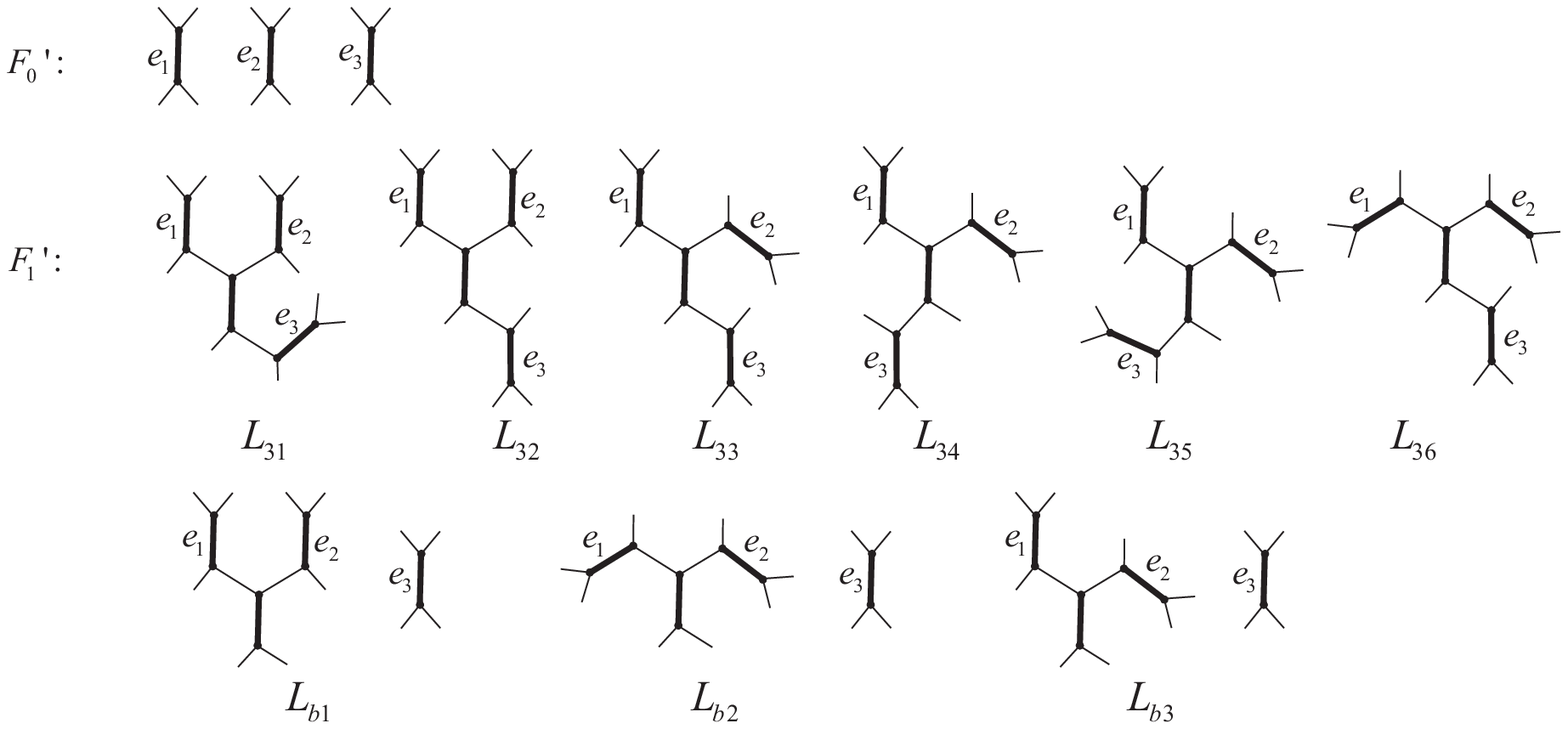}
\caption{\label{seed-case3}{\small $F_0'$ has three components.}}
\end{figure}

\emph{Subcase 3.1.} $F$ has a subgraph $L_{b1}$.

Let $F_1':=L_{b1}$ and $F_1'':=F-F_1'$. Clearly, $F_1''$ has a $1$-degree vertex that is adjacent to one end of $e_1$ and one end of $e_2$. Then $F$ has a subgraph $F_2'$ as depicted in Fig. \ref{seed-case3-1}, or has a generalized patch $L_{31}$ or $L_{32}$ (see Fig. \ref{seed-case3}).  Let $F_2'':=F-F_2'$. If $F_2''$ has not $1$-degree vertices, then it also has not pendent pentagons by Lemma \ref{main-lem} $(ii)$ since the distance-array of $F_2'-V(e_3)$ is $[13231323]$. By Theorem \ref{Operation3}, $F_2''$ consists of two patches $U_2^1, V_2^1\in \mathcal{D}$ connecting by only an edge $x_2^1$.
\begin{figure}[htbp!]
\centering
\includegraphics[height=4.5cm]{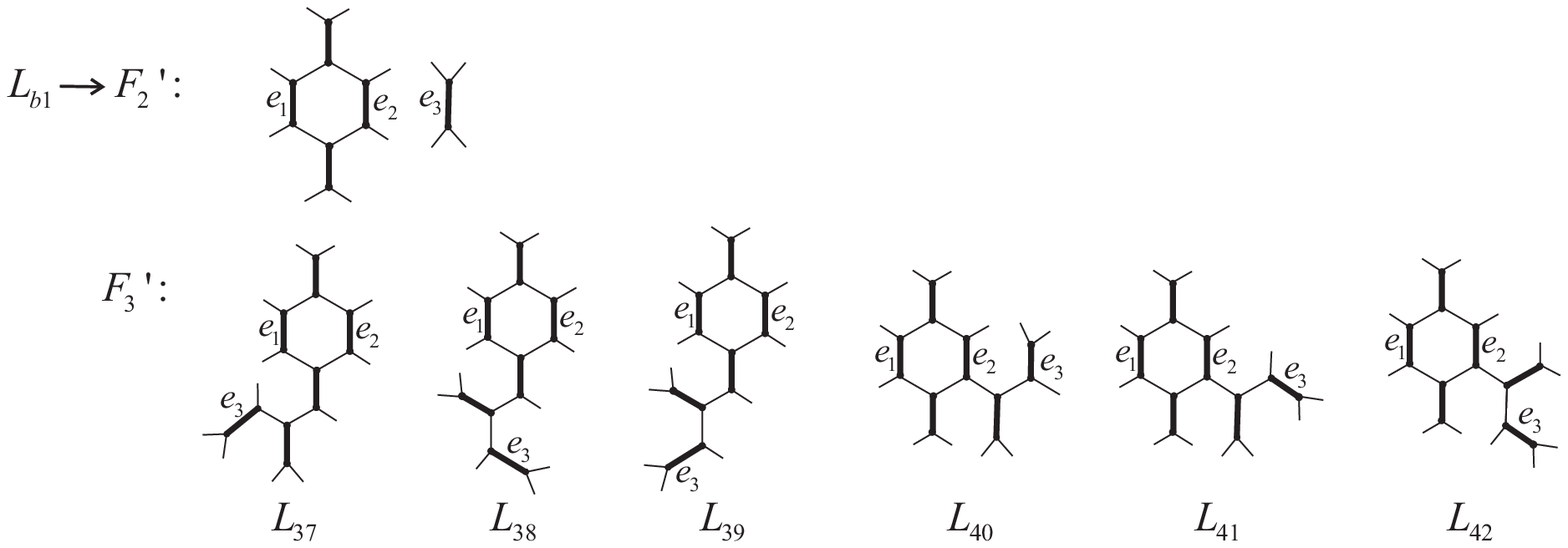}
\caption{\label{seed-case3-1}{\small $F_1'$ is isomorphic to $L_{b1}$.}}
\end{figure}
By Lemma \ref{main-lem} $(i)$, $F[F_2''\cup V(e_3)]$ is a patch that has the structure as shown in Fig. \ref{full_structure1}, and
the edges in $\triangledown_F(L_{b1}-V(e_3))$ and $\triangledown_F(F[F_2''\cup V(e_3)])$ merge successively along the boundaries of $L_{b1}-V(e_3)$ and $F[F_2''\cup V(e_3)]$, respectively.
So the  structures of $F[V(U_2^1)\cup V(F_2')]$ and $F[V(V_2^1)\cup V(F_2')]$ are as depicted in Fig. \ref{structure1} $(b)$.
The distance-arrays of $U_2^1$ and $V_2^1$ are both $[a_1\wedge a_2a_32a_5a_6a_7]$, where $a_1, a_2, a_5, a_7=2$ or $3$ and $a_3, a_6=3$ or $4$.
It is easy to check that the distance-arrays $[2\wedge232232]$ of $D_{05}$, $[3\wedge232232]$ of $D_{08}$, $[2\wedge242233]$ of $D_{09}$, $[2\wedge332233]$ of $D_{12}$, $[3\wedge242242]$ of $D_{13}$, $[2\wedge242333]$ of $D_{18}$, $[2\wedge342243]$ of $D_{19}$ satisfy the above condition of $U_2^1$. Similarly, we have the same conclusion for $V_2^1$. So $U_2^1, V_2^1\in\{D_{05}, D_{08}, D_{09}, D_{12}, D_{13}, D_{18}, D_{19}\}$. Hence $F\cong W_i$, $i\in\{4, \ldots, 26\}$ (see Fig. \ref{fullerene1}).
If $F_2''$ has a $1$-degree vertex, then this vertex connects the two components of $F_2'$. So $F$ has a patch $L_i$, $i\in\{37, \ldots, 42\}$.

\emph{Subcase 3.2.} $F$ has a subgraph $L_{b2}$.

By the above discussions, we suppose that any two edges in $S$ can not be as $\{e_1, e_2\}$ in $L_{b1}-V(e_3)$.
Let $F_1':=L_{b2}$ and $F_1'':=F-F_1'$.
If $F_1''$ has not $1$-degree vertices, then it also has no pendent pentagons by Lemma \ref{main-lem} (ii) since the distance-array of $L_{b2}-V(e_3)$ is $[12321414]$. By Theorem \ref{Operation3}, $F_1''$ consists of two patches $U_1^2, V_1^2\in \mathcal{D}$ connecting by only one edge $x_1^2$.
By Lemma \ref{main-lem} (i), the structures of $F[V(U_1^2)\cup V(F_1')]$ and $F[V(V_1^2)\cup V(F_1')]$ are as depicted in Fig. \ref{structure1} $(c), (d)$ respectively. So the distance-array of $U_1^2$ is $[a_1\wedge a_2a_322a_62]$, of $V_1^2$ is $[a_1'\wedge a_2'a_3'2a_5'a_6'a_7']$, where $a_1, a_2=2$ or $3$, $a_3=3$ or $4$, $a_6=4$ or $5$, $a_1', a_2', a_6'=2$ or $3$, $a_3', a_5', a_7'=3$ or $4$. We can check that the distance-array $[3\wedge242242]$ of $D_{13}$ satisfies the condition of $U_1^2$, and the distance-arrays $[2\wedge232323]$ of
$D_{08}$, $[2\wedge242333]$ of $D_{18}$, $[2\wedge242424]$ of $D_{29}$ satisfy the condition of $V_1^2$.
So $F\cong W_{27}, W_{28},$ or $W_{29}$ (see Fig. \ref{fullerene1}).
\begin{figure}[htbp!]
\centering
\includegraphics[height=5.0cm]{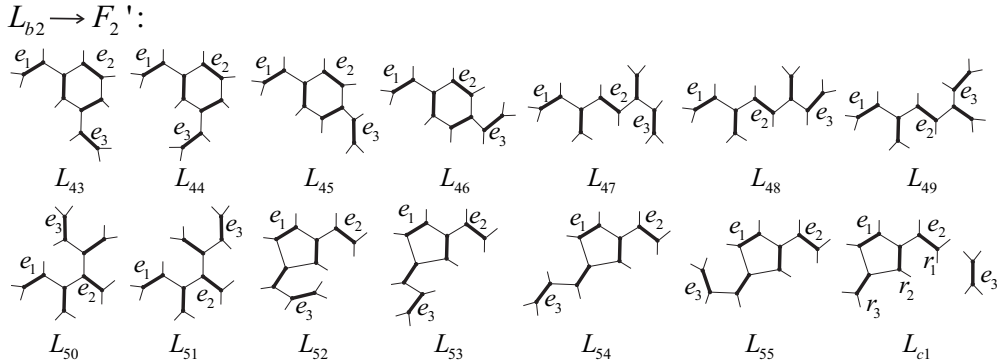}
\caption{\label{seed-case3-2.1}{\small $F_1'$ is isomorphic to $L_{b2}$.}}
\end{figure}

If $F_1''$ has a $1$-degree vertex, then this vertex connects the two components of $F_1'$ or  has two adjacent vertices in $L_{b2}-V(e_3)$. For the first case, $F$ has a generalized patch $L_i$, $i\in \{43, \ldots, 51\}$. For the second case, $F$ has a subgraph $L_{c1}$ or a generalized patch $L_i$, $i\in\{52, 53, 54, 55\}$.

Now, we suppose that $F$ has a subgraph $L_{c1}$. Let $F_2':=L_{c1}$ and $F_2'':=F-F_2'$. We notice that $F_2''$ has a unique perfect matching $M\cap E(F_2'')$.

If $F_2''$ has not $1$-degree vertices, then $F_2''$ is connected by Lemma \ref{3-4-edge-cut} and Proposition \ref{cyclic-6-edge-cut}.
If $F_2''$ has a pendent pentagon, say $P_x$, then by Lemma \ref{main-lem} (ii) the three edges in $E(P_x, L_{c1}-V(e_3))$ are three consecutive edges in $\triangledown_F(P_x)$ along the boundary of $P_x$. So the three edges $r_1, r_2, r_3$ in $L_{c1}$ (see Fig. \ref{seed-case3-2.1}) are those three edges in $E(P_x, L_{c1}-V(e_3))$. Hence
$F$ has a generalized patch $L_{56}$ or $L_{57}$.
\begin{figure}[htbp!]
\centering
\includegraphics[height=10cm]{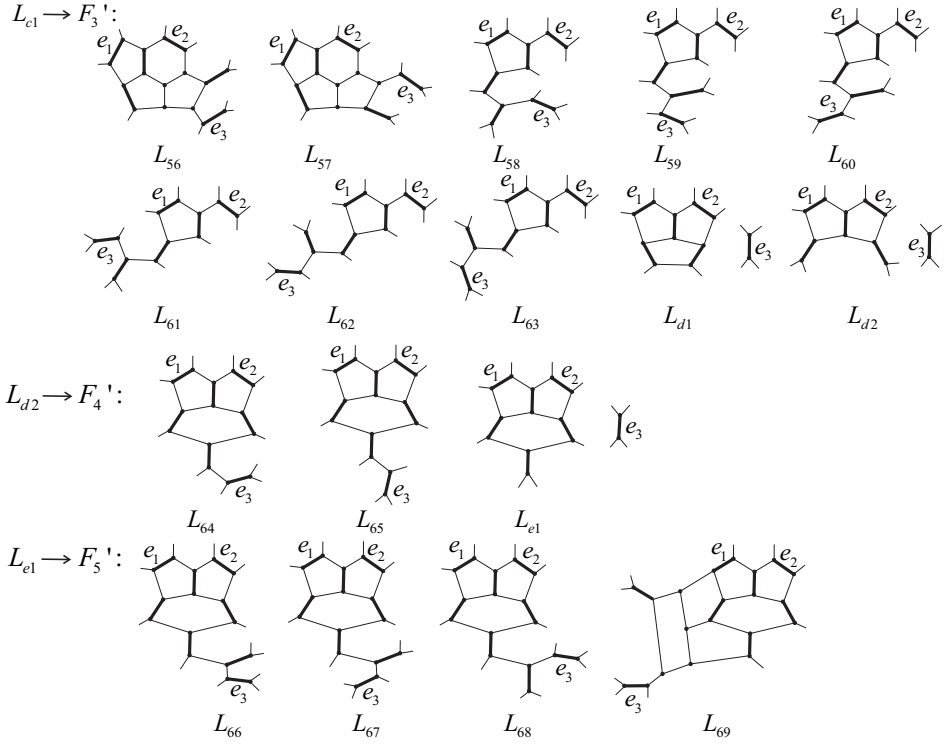}
\caption{\label{seed-case3-2.2}{\small $F_2'$ is isomorphic to $L_{c1}$.}}
\end{figure}

If $F_2''$ has no pendent pentagons, then $F_2''$ consists of two patches $U_2^2, V_2^2\in \mathcal{D}$ connecting by only an edge $x_2^2$ by Theorem \ref{Operation3}. By Lemma \ref{main-lem} (i), the structures of $F[V(U_2^2)\cup V(F_2')]$ and $F[V(V_2^2)\cup V(F_2')]$ are as depicted in Fig. \ref{structure2}.
So the distance-array of $V_2^2$ is $[2\wedge a_2'a_3'2a_5'2a_7']$, where $a_2', a_5'=2$ or $3$, $a_3'=3$ or $4$, $a_7'=4$ or $5$. We can check that the distance-array of each patch in $\mathcal{D}$ does not satisfy the condition of $V_2^2$.
So $V_2^2\notin \mathcal{D}$, a contradiction.
\begin{figure}[htbp!]
\centering
\includegraphics[height=3.2cm]{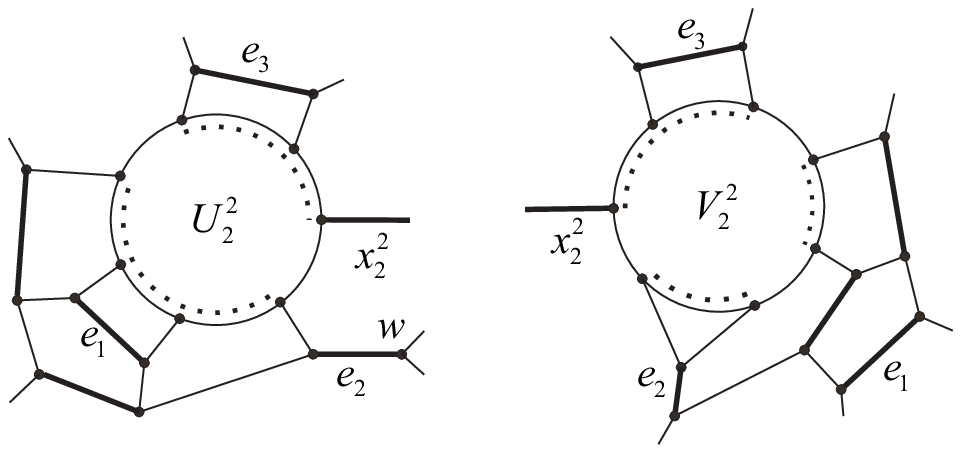}
\caption{\label{structure2}{\small}}
\end{figure}

If $F_2''$ has a $1$-degree vertex, then this vertex connects the two components of $F_2'$ or not. For the first case, $F$ has a generalized patch $L_i$, $i\in \{58, \ldots, 63\}$, or has a generalized patch $L_j$, $j\in \{43, 45, 46, \ldots, 51\}$. For the second case, $F$ has a generalized patch $L_i$, $i\in \{54, 55\}$, or has a subgraph $L_{d1}$ or $L_{d2}$ (see Fig. \ref{seed-case3-2.2}).
If $F$ has a subgraph $L_{d1}$, then by Theorem \ref{Operation3} and Lemma \ref{main-lem} (ii) $F-V(L_{d1})$ has a $1$-degree vertex since $|\triangledown_F(L_{d1})|=10$ and the distance-array of $L_{d1}-V(e_3)$ is $[232323]$. Clearly, this $1$-degree vertex
connects the two components of $L_{d1}$. So $F$ has a generalized patch $L_i$, $i\in \{48, 49, 50, 51, 54, 55\}$.

If $F$ has a subgraph $L_{d2}$, then $F-V(L_{d2})$ has a $1$-degree vertex that has two adjacent vertices in $L_{d2}-V(e_3)$. So $F$ has a subgraph $L_{e1}$ (see Fig. \ref{seed-case3-2.2}), or has a generalized patch $L_{64}$ or $L_{65}$. For the subgraph $L_{e1}$, $F-V(L_{e1})$ has a $1$-degree vertex or a pendent pentagon by Theorem \ref{Operation3} and Lemma \ref{main-lem} (i) since $|\triangledown_F(L_{e1})|=12$ and the distance-array of $L_{e1}-V(e_3)$ is $[13323233]$.
So we can check that $F$ has a generalized patch $L_i$, $i\in \{66, 67, 68, 69\}$, or has a generalized patch $L_i$, $i\in \{48, 49, 50, 51, 61, 62\}$.

\emph{Subcase 3.3.} $F$ has a subgraph $L_{b3}$.
\begin{figure}[htbp!]
\centering
\includegraphics[height=7.5cm]{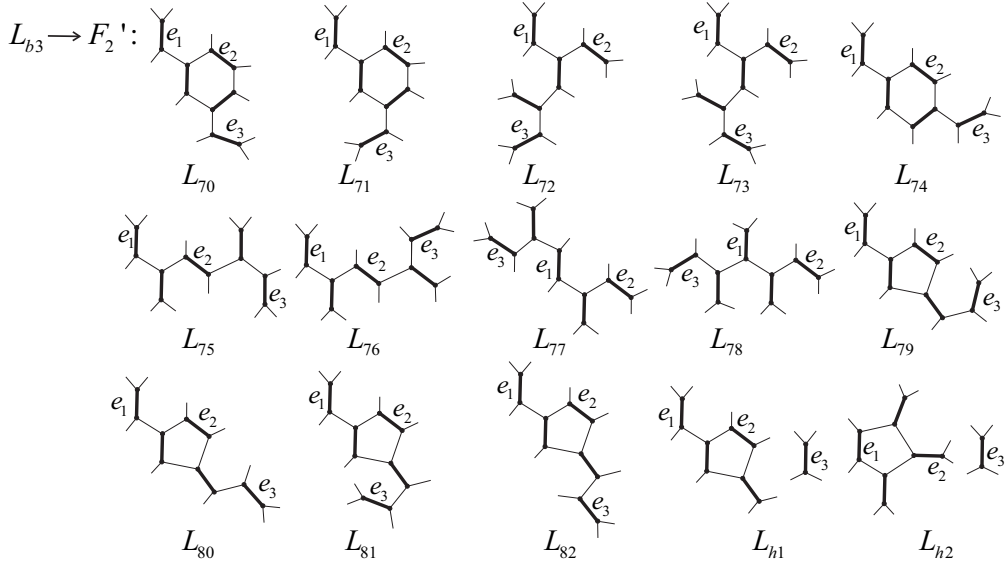}
\caption{\label{seed-case3-3}{\small $F_1'$ is isomorphic to $L_{b3}$.}}
\end{figure}

By the above discussions, we suppose that any two edges in $S$ can not be as $\{e_1, e_2\}$ in $L_{b1}-V(e_3)$ or $L_{b2}-V(e_3)$.
Let $F_1':=L_{b3}$ and $F_1'':=F-F_1'$.

If $F_1''$ has not $1$-degree vertices, then by Lemma \ref{main-lem} (ii) $F_1''$ has no pendent pentagons since the distance-array of $L_{b3}-V(e_3)$ is $[12314124]$.
By Theorem \ref{Operation3}, $F_1''$ consists of two patches $U_1^3, V_1^3\in\mathcal{D}$ connecting by only an edge $x_1^3$. By Lemma \ref{main-lem} (i), the structures of $U_1^3$ and $V_1^3$ are as depicted in Fig. \ref{structure3} $(a), (b)$.
So the distance-array of $U_1^3$ is $[2\wedge a_2a_322a_6a_7]$ and of $V_1^3$ is $[2\wedge a_2'a_3'2a_5'2a_7']$, where $a_2, a_7=2$ or $3$, $a_3=3$ or $4$, $a_6=4$ or $5$, and $a_2'=2$ or $3$, $a_3', a_5'=3$ or $4$, $a_7'=4$ or $5$.
We can check that the distance-array $[2\wedge342243]$ of $D_{19}$ satisfies the condition of $U_1^3$, and the distance-array $[2\wedge242424]$ of $D_{29}$ satisfies the condition of $V_1^3$.
So $U_1^3\cong D_{19}$ and $V_1^3\cong D_{29}$.
Clearly, $F[V(U_1^3)\cup\{z\}]\cong P^3$ (see Fig. \ref{sub-patches}) and $\triangledown_F(F[V(U_1^3)\cup\{z\}])$ is a non-degenerate cyclic $6$-edge-cut of $F$. By Lemma \ref{determine-nanotube-type(4,2)}, $F$ is a nanotube fullerene of type $(4, 2)$, a contradiction.
\begin{figure}[htbp!]
\centering
\includegraphics[height=3.1cm]{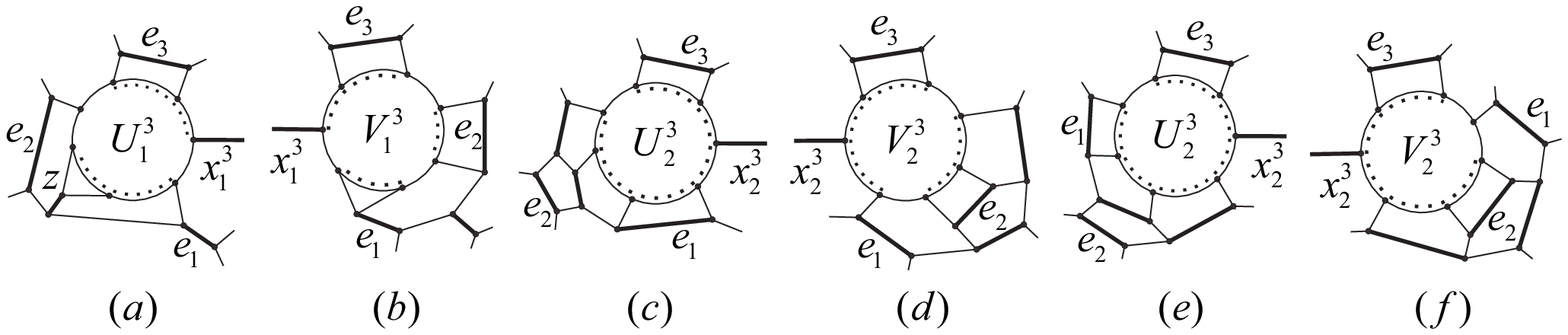}
\caption{\label{structure3}{\small}}
\end{figure}

If $F_1''$ has a $1$-degree vertex, then this vertex connects the two components of $F_1'$ or not. For the first case, $F$ has a generalized patch $L_i$, $i\in\{70, \ldots, 78\}$  or a generalized patch $L_{33}$ or $L_{35}$. For the second case, $F$ has a subgraph $L_{h1}$ or $L_{h2}$ (see Fig. \ref{seed-case3-3}), or has a generalized patch $L_j$, $j\in\{79, 80, 81, 82\}$, or a generalized patch $L_r$, $r=33, 34,$ or $35$.
\begin{figure}[htbp!]
\centering
\includegraphics[height=3.2cm]{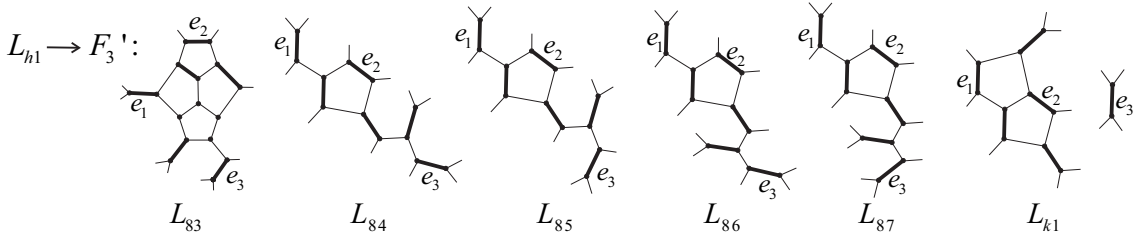}
\caption{\label{seed-case3-3.1}{\small $F_2'$ is isomorphic to $L_{h1}$.}}
\end{figure}

If $F$ has a subgraph $L_{h1}$, then let $F_2':=L_{h1}$ and $F_2'':=F-F_2'$. We note that $F_2''$ has a unique perfect matching. If $F_2''$ has not $1$-degree vertices, then $F_2''$ has a pendent pentagon or not since the distance-array of $L_{h1}-V(e_3)$ is $[13241233]$.
For the first case, as the discussion of $L_{c1}$ (see Fig. \ref{seed-case3-2.1}), we can show that $F$ has a generalized patch $L_{83}$.
For the second case,  $F_2''$ consists of two patches $U_2^3, V_2^3\in \mathcal{D}$ connecting by only an edge $x_2^3$ by Theorem \ref{Operation3}.
By Lemma \ref{main-lem} (i), the structures of $U_2^3$ and $V_2^3$ are as depicted in Fig. \ref{structure3} $(c)(d)$, or as Fig. \ref{structure3} $(e)(f)$.
If the structures of $U_2^3$ and $V_2^3$ are
as depicted in Fig. \ref{structure3} $(c)(d)$, the distance-array of $U_2^3$ is $[a_1\wedge a_2a_32a_5a_6a_7]$ and of $V_2^3$ is $[a_1'\wedge a_2'a_3'2a_5'a_6'2]$, where $a_1, a_2, a_5, a_6=2$ or $3$, $a_3, a_7=3$ or $4$, and $a_1', a_2', a_5'=2$ or $3$, $a_3', a_6'=3$ or $4$.
We can check that the distance-arrays $[2\wedge232223]$ of $D_{05}$, $[2\wedge232323]$ of $D_{08}$, $[2\wedge242233]$ of $D_{09}$, $[2\wedge332233]$ of $D_{12}$, $[2\wedge242333]$ of $D_{18}$ satisfy the condition of $U_2^3$,
\begin{figure}[htbp!]
\centering
\includegraphics[height=21.2cm]{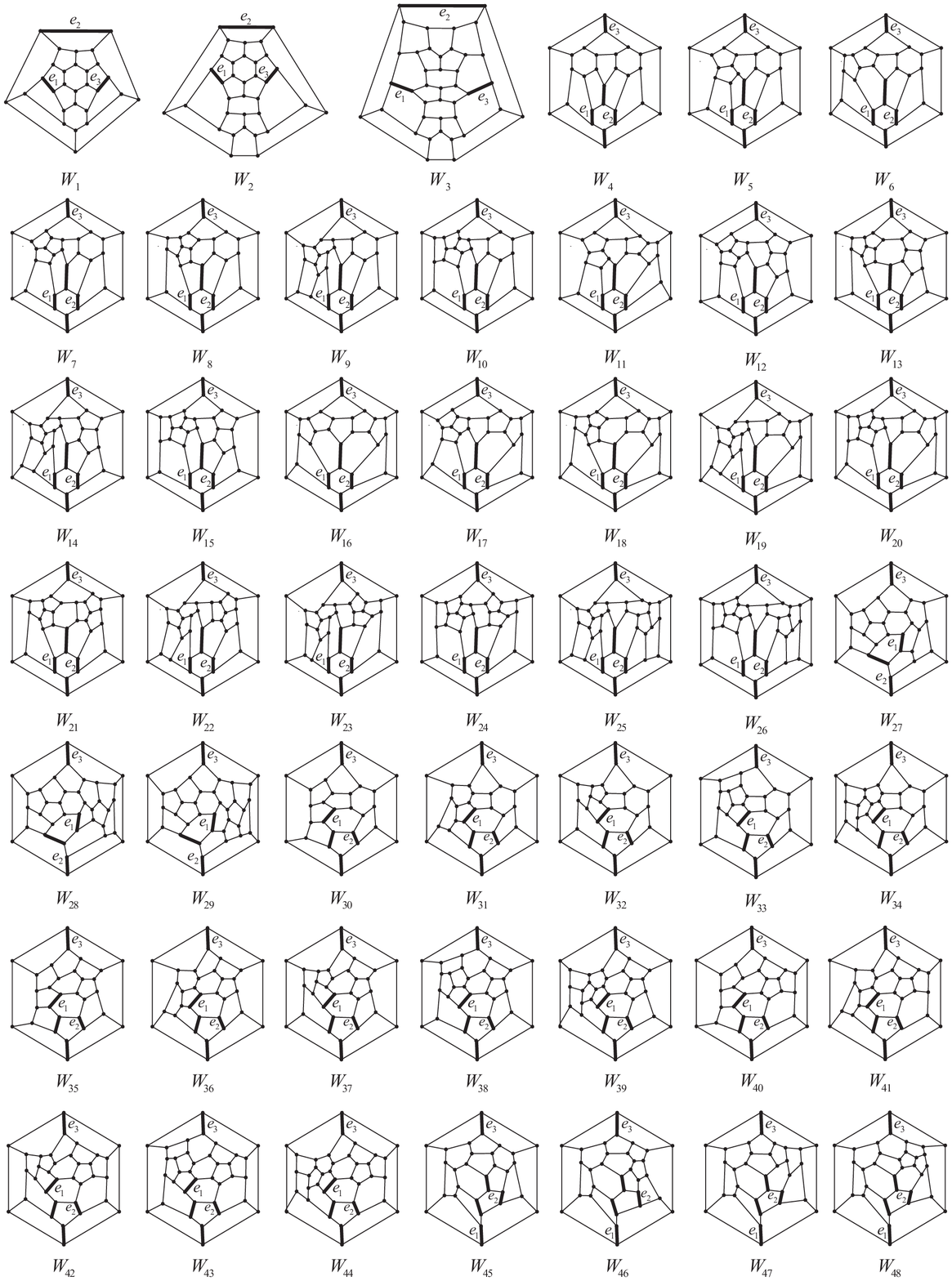}
\caption{\label{fullerene1}{\small}}
\end{figure}
and the distance-arrays $[2\wedge232232]$ of $D_{05}$, $[3\wedge232232]$ of $D_{08}$, $[3\wedge242242]$ of $D_{13}$ satisfy the condition of $V_2^3$.
Then $F\cong W_i$, $i\in\{30, \ldots, 44\}$ (see Fig. \ref{fullerene1}).
If the structures of $U_2^3$ and $V_2^3$ are as depicted in Fig. \ref{structure3} $(e)(f)$, we can similarly show that
$U_2^3=D_{08}$ or $D_{18}$, and $V_2^3=D_{05}, D_{08}, D_{09}, D_{12}, D_{13}$ or $D_{19}$. So
$F\cong W_i$, $i\in\{45, \ldots, 56\}$  (see Fig. \ref{fullerene1}, \ref{fullerene2}).

If $F_2''$ has a $1$-degree vertex, then this vertex connects the two components of $F_2'$ or not. For the first case, $F$ has a generalized patch $L\in\{L_{84}, L_{85}, L_{86}, L_{87}\}$ or a generalized patch $L_i$, $i\in\{33, 35, 72, 76, 77, 78\}$.
For the second case, $F$ has a subgraph $L_{k1}$ (see Fig. \ref{seed-case3-3.1}), or a generalized patch $L_j$, $j\in\{33, 34, 35\}$.
\begin{figure}[htbp!]
\centering
\includegraphics[height=7.0cm]{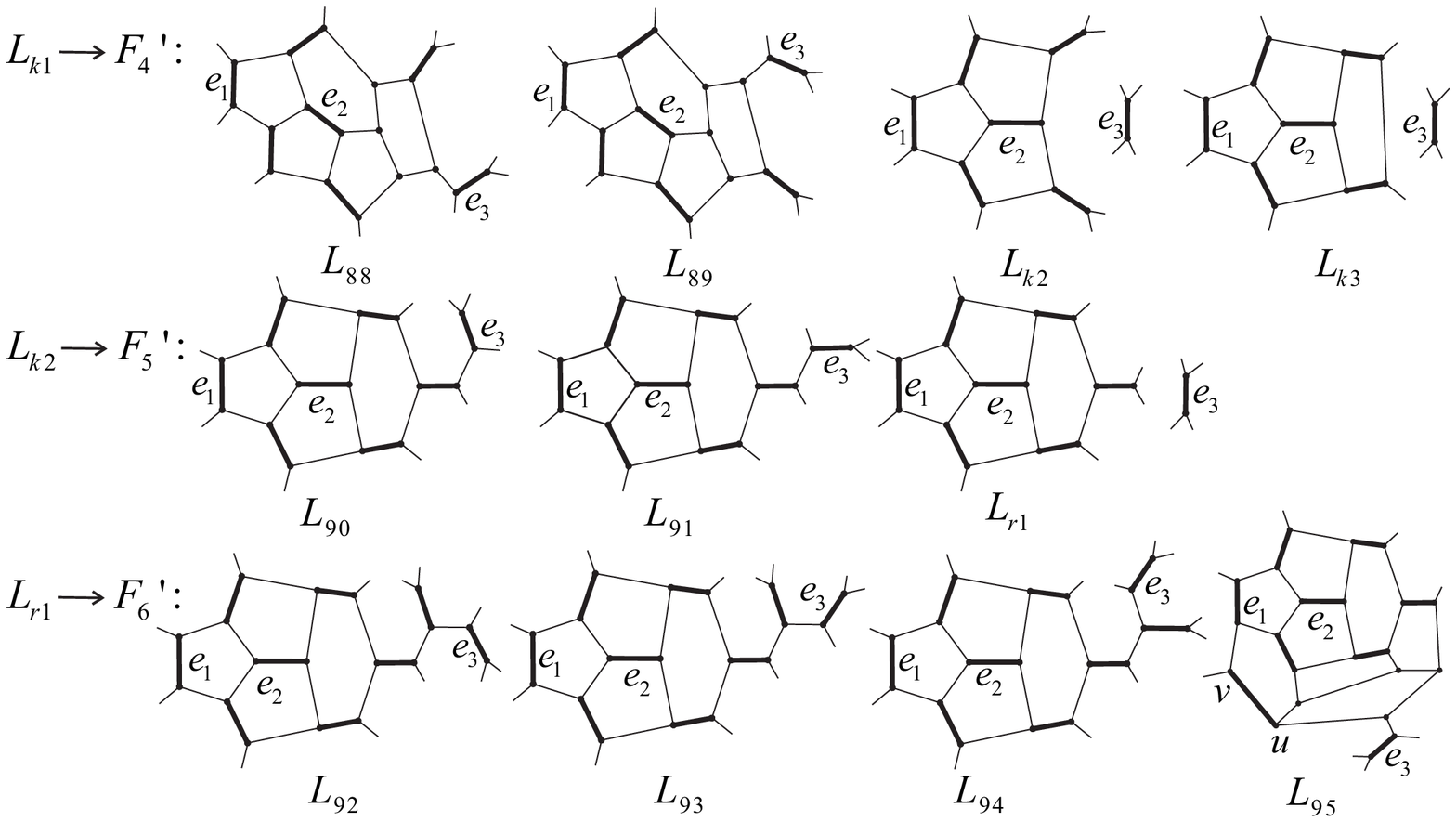}
\caption{\label{seed-case3-3.2}{\small $F_3'$ is isomorphic to $L_{k1}$.}}
\end{figure}
If $F$ has a subgraph $L_{k1}$, then let $F_3':=L_{k1}$ and $F_3'':=F-F_3'$.
If $F_3''$ has not $1$-degree vertices, then we claim that it has a pendent pentagon by Theorem \ref{Operation3} and Lemma \ref{main-lem} $(i)$ since the distance-array of $L_{k1}-V(e_3)$ is $[13323143]$.
By Lemma \ref{main-lem} (ii), $F$ has a generalized patch $L_{83}$, or has a generalized patch $L_{88}$ or $L_{89}$ (see Fig. \ref{seed-case3-3.2}).
If $F_3''$ has a $1$-degree vertex, then this vertex connects the two components of $F_3'$ or not. For the first case, $F$ has a generalized patch $L_i$, $i\in\{70, 71, 72, 73, 77, 78, 84, 85, 86, 87\}$.
For the second case, $F$ has a generalized patch $L_{81}$ or $L_{82}$, or has a subgraph $L_{k2}$ or $L_{k3}$ (see Fig. \ref{seed-case3-3.2}). If $F$ has a subgraph $L_{k3}$, then it has a generalized patch $L_j$, $j\in\{72, 78, 83, 86, 87\}$.
If $F$ has a subgraph $L_{k2}$, then $F-L_{k2}$ has a $1$-degree vertex that is adjacent to two vertices in $L_{k2}-V(e_3)$. So $F$ has a subgraph $L_{r1}$, or a generalized patch $L_{90}$, $L_{91}$. For the subgraph $L_{r1}$,
if $F_4'':=F-L_{r1}$ has a $1$-degree vertex, then this vertex connects the two components of $L_{r1}$. So $F$ has a generalized patch $L_i$, $i\in\{72, 78, 87, 92, 93, 94\}$.
Next, we suppose that $F_4''$ has no $1$-degree vertices. Then $F_4''$ has a pendent pentagon or not. If $F_4''$ has a pendent pentagon, then we can check that $F$ has a generalized patch $L_{95}$. If $F_4''$ has no pendent pentagons, then by Theorem \ref{Operation3} and Lemma \ref{main-lem} (i), $F_4''$ consists of two patches $U_3^3$, $V_3^3\in\mathcal{D}$, and the structures of $U_3^3$ and $V_3^3$ are as depicted in Fig. \ref{structure4} $(a)$. We can check that the distance-arrays $[2\wedge232223]$ and $[2\wedge232232]$ of $D_{05}$, $[2\wedge232323]$ of $D_{08}$, $[2\wedge242233]$ of $D_{09}$, $[2\wedge332233]$ of $D_{12}$, $[2\wedge242333]$ of $D_{18}$ satisfy the condition of $U_3^3$ (resp. $V_3^3$).
So $F\cong W_i$, $i\in\{57, \ldots, 76\}$ (see Fig. \ref{fullerene2}).
\begin{figure}[htbp!]
\centering
\includegraphics[height=4cm]{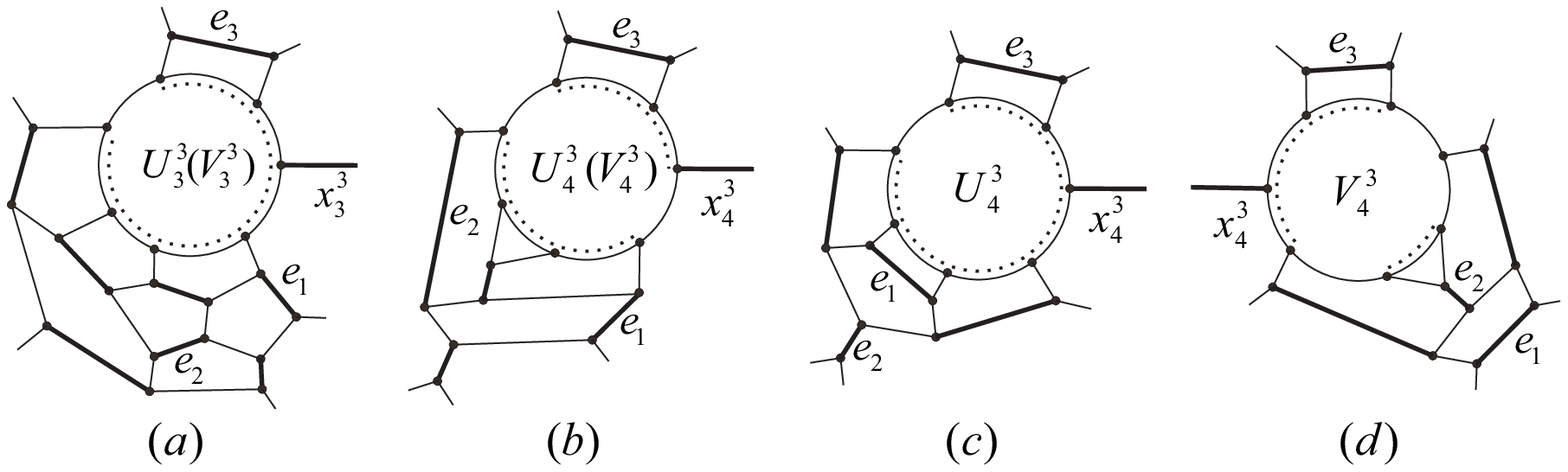}
\caption{\label{structure4}}
\end{figure}

Now, we suppose that $F$ has a subgraph $L_{h2}$ (see Fig. \ref{seed-case3-3}).
If $F-L_{h2}$ has no $1$-degree vertices, then $F-L_{h2}$ has no pendent pentagons by Lemma \ref{main-lem} (ii) since the distance-array of $L_{h2}-V(e_3)$ is $[14132314]$.
By Theorem \ref{Operation3}, $F-L_{h2}$ consists of two patches $U_4^3, V_4^3\in\mathcal{D}$ connecting by only one edge. Since the distance-array of $L_{h2}-V(e_3)$ is $[14132314]$, by Lemma \ref{main-lem} (i), the structures of $U_4^3$ and $V_4^3$
\begin{figure}[htbp!]
\centering
\includegraphics[height=14cm]{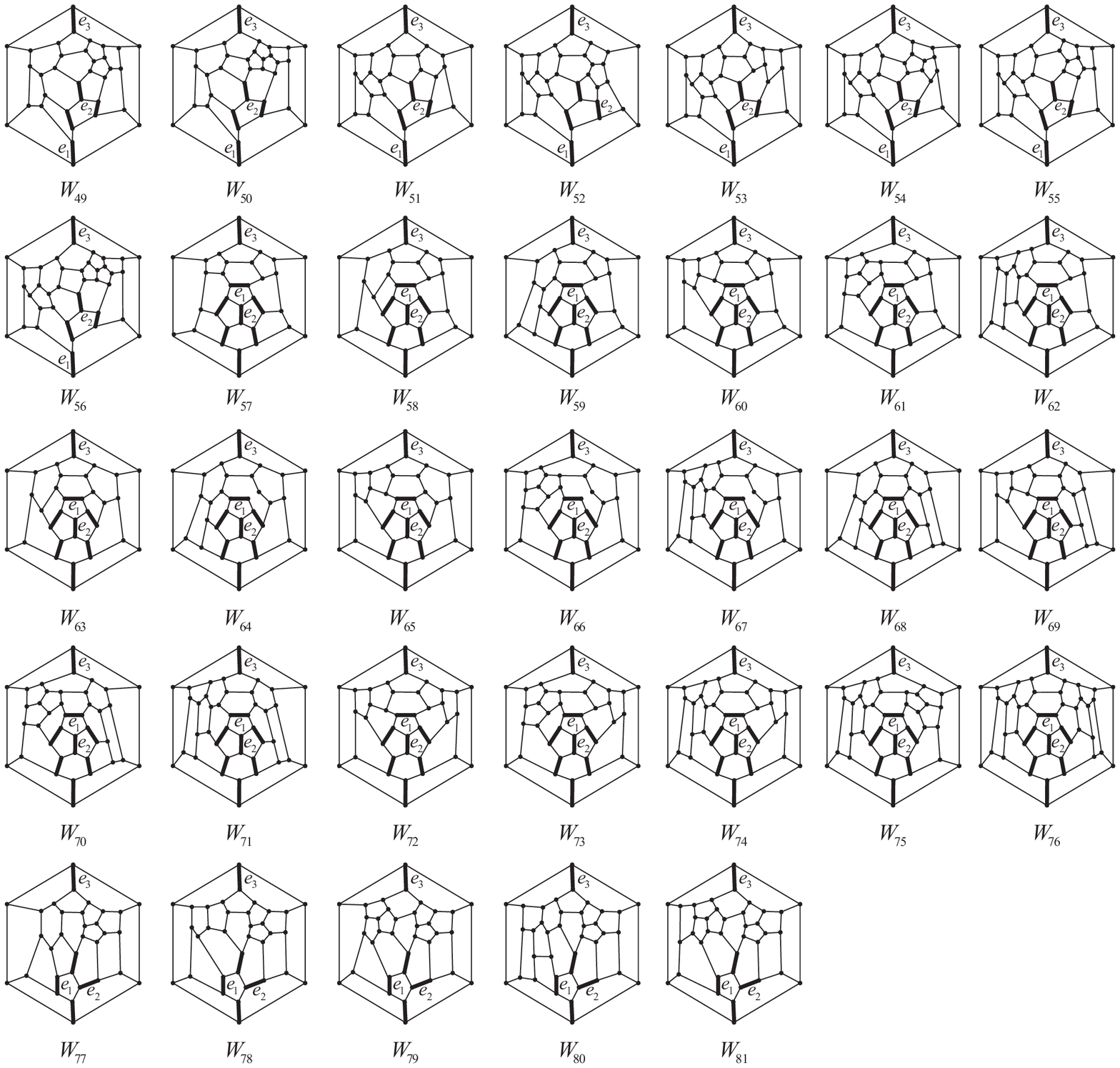}
\caption{\label{fullerene2}{\small}}
\end{figure}
are as depicted in Fig. \ref{structure4} (b) or as depicted in Fig. \ref{structure4} (c) (d).
If the structures of $U_4^3$ and $V_4^3$ are both as depicted in Fig. \ref{structure4} $(b)$, then we can check that only the distance-array $[2\wedge342243]$ of $D_{19}$ satisfies the condition of $U_4^3$ (resp. $V_4^3$). But then we will obtain a $7$-cycle, a contradiction. If the structures of $U_4^3$ and $V_4^3$ are as depicted in Fig. \ref{structure4} $(c), (d)$, then we can check that only the distance-array $[3242242]$ of $D_{13}$ satisfies the condition of $V_4^3$.
It implies that the distance-array of $U_4^3$ is $[2\wedge a_2a_32a_5a_6a_7]$, where $a_2, a_5, a_7=2$ or $3$, $a_3, a_6=3$ or $4$. We can check that the distance-arrays $[2\wedge232232]$ of $D_{05}$, $[2\wedge242233]$ of $D_{09}$, $[2\wedge332233]$ of $D_{12}$, $[2\wedge242333]$ of $D_{18}$, $[2\wedge342243]$ of $D_{19}$ satisfy the condition of $U_4^3$. So $U_4^3=D_{05}, D_{09}, D_{12}, D_{18}$ or $D_{19}$.
Hence $F\cong W_i$, $i\in\{77, \ldots, 81\}$.
If $F-L_{h2}$ has a $1$-degree vertex, then this vertex connects the two components of $L_{h2}$ or not.
For the first case, we can show that $F$ has a generalized patch $L_j$, $j\in\{70, 71, 72, 73, 74, 78\}$. For the second case, we can show that $F$ has a subgraph $L_{k1}$ (see Fig. \ref{seed-case3-3.1}), or has a generalized patch $L_i$, $i\in\{79, 80, 81, 82\}$. Note that we have considered $L_{k1}$ in the above paragraph.

Hence $F\cong W_i$, $i\in\{1, \ldots, 81\}$ (see Fig. \ref{fullerene1}, \ref{fullerene2}), or has a generalized patch $L\in \mathcal{L}$. It is easy to check that $S\subset E(L)$ forces a perfect matching of $L$ in the sense of $F$.
\end{proof}

\section{Constructing all fullerenes with $f(F)=3$}

Up to now, we know that the nanotube fullerenes of type $(4, 2)$ and fullerenes $W_1, \ldots, W_{81}$ (see Fig. \ref{fullerene1}, \ref{fullerene2}) have the minimum forcing number three.
In the sequel, we want to know all the other fullerene graphs with the minimum forcing number three.
Starting from a generalized patch $L\in\mathcal{L}$, our idea is to expand $L$ to a larger
generalized patch $L'$ such that $S=\{e_1, e_2, e_3\}$ forces a perfect matching of $L'$ in the sense of $F$.
By a serious of such expansions, we will finally obtain a fullerene with $S$ being a forcing set.
\begin{figure}[htbp!]
\centering
\includegraphics[height=5.8cm]{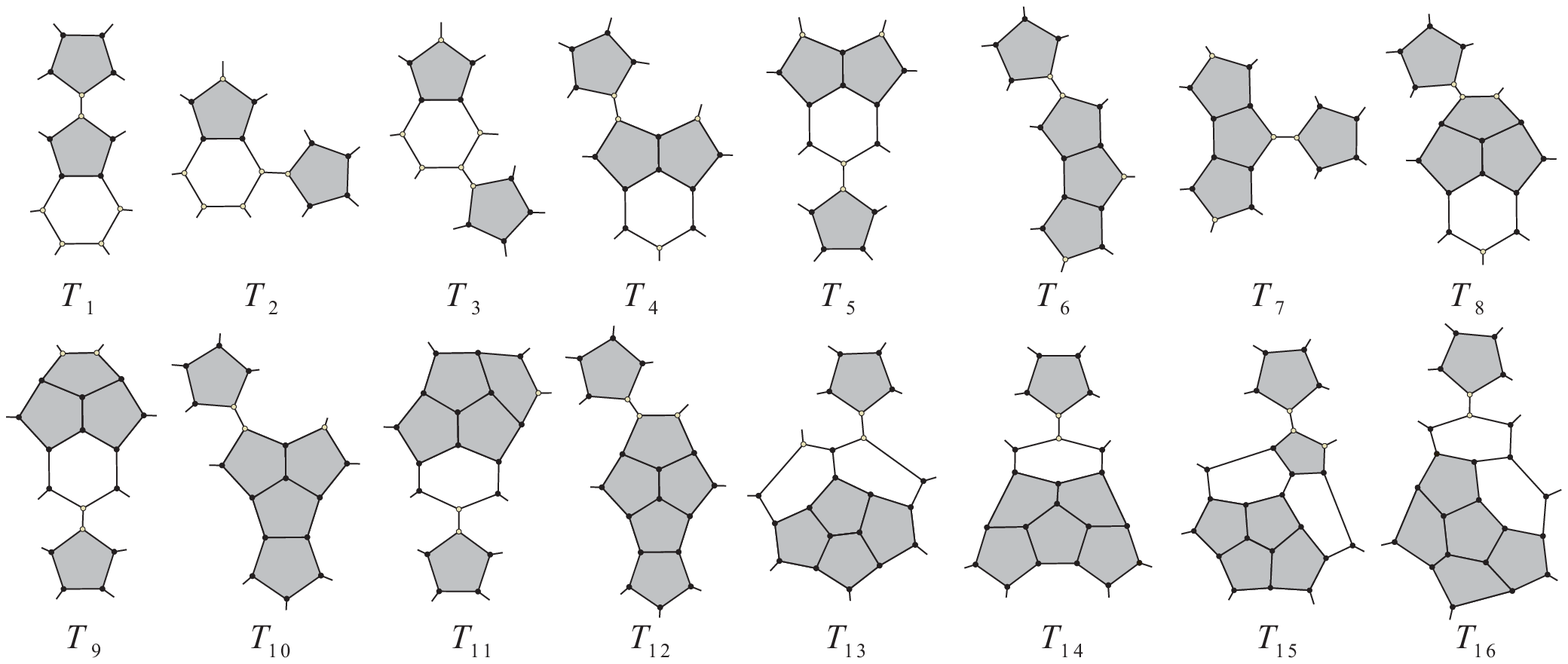}
\caption{\label{terminal10}{\small $|\triangledown_F(F'')|=10$.}}
\end{figure}

Let $L$ be a generalized patch of fullerene $F$ and $[a_1\cdots a_k]$ a distance-array of $L$.
As the operations defined in Ref. \cite{YangQ}, we define the following operations on $L$:

$(O_1)$ If $k\geq4$, $a_i\in\{4, 5\}$ and $a_{i-1}, a_{i+1}\in\{1, 2, 3, 4\}$, then let $t_i$ and $t_{i+1}$ be incident with a new vertex $u$. Add another
new vertex $v$, an edge $uv$, and attach two half edges to $v$. A distance-array for the
resulting generalized patch $L'$ would be $[a_1, \ldots, a_{i-2}, a_{i-1}+2, 1, a_{i+1}+2, a_{i+2}, \ldots, a_k]$.

$(O_2)$ If $k\geq4$, $a_i\in\{5, 6\}$ and $a_{i-1}+a_{i+1}\leq6$, then let the half edges $t_i$ and $t_{i+1}$ merge into one edge. A distance-array for the resulting generalized patch $L'$ would be $[a_1, \ldots, a_{i-2}, a_{i-1}+a_{i+1}, a_{i+2}, \ldots, a_k]$.

$(O_3)$ If $k=2$ and $a_1, a_2\in\{5, 6\}$, then we merge the half edges $t_1$ and $t_2$ into one edge.
\begin{figure}[htbp!]
\centering
\includegraphics[height=19cm]{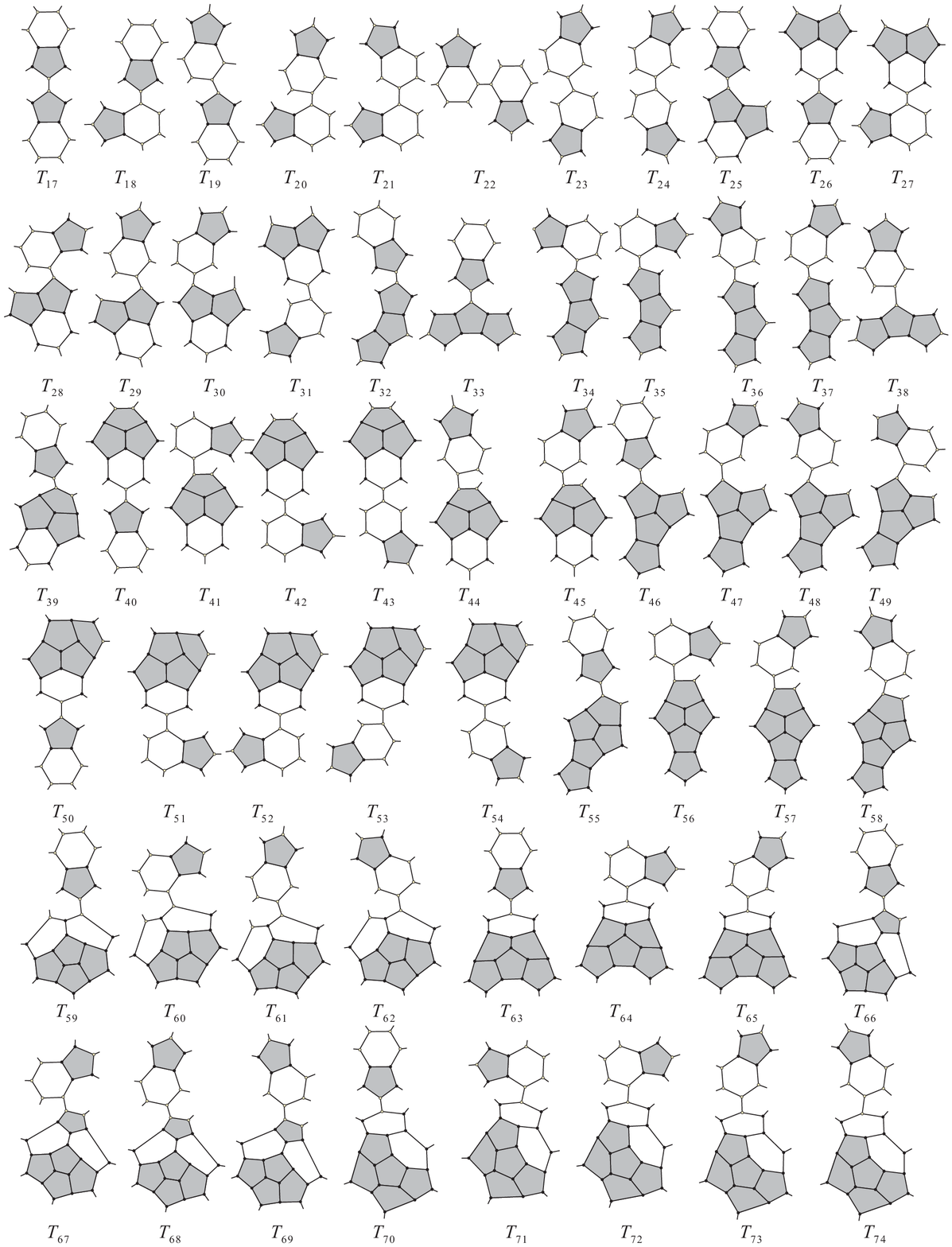}
\caption{\label{terminal12I}{\small $|\triangledown_F(F'')|=12$.}}
\end{figure}
The distance-array of the
resulting graph $L'$ is the empty distance-array $[]$.

$(O_4)$ If $k=8$, $a_i, a_{i+4}\in\{1, 2\}$ for some $i$, and all other $a_j$ are $3$ or $4$, then we connect each half-edge of $L$ to a $2$-degree vertex in $J_2$ (see Fig. \ref{sub-patches}) in the only admissible way. The distance-array of the resulting graph $L'$ is the empty distance-array $[]$.

$(O_5)$ If $k=10$, $a_i\in\{1, 2, 3, 4\}$ for $i\in\{1, 2, \ldots, 10\}$, then we check that whether the ten half edges of $L$ and the ten half edges
of some generalized patch $PP$ or $T_r$, $r\in\{1, 2, \ldots, 16\}$ (see Fig. \ref{terminal10})
can merge such that the resulting ten new faces are pentagons or hexagons,
that is, we check that whether the sum $[a_{i_1}+b_1, \ldots, a_{i_{10}+b_{10}}]$ of some distance-array $[a_{i_1}, \ldots, a_{i_{10}}]$ of $L$ and the min-distance-array $[b_1, \ldots, b_{10}]$ of $PP$ or $T_r$, $r\in\{1, 2, \ldots, 16\}$ has $a_{i_s}+b_s=5$ or $6$ for each $s=1, \ldots, 10$.
We record all efficient cases and merge the ten edges according to the various efficient cases, respectively.
Here, we may obtain different fullerenes.  However, the distance-array
corresponding to the resulting graph $L'$ is the empty distance-array $[]$.
\begin{figure}[htbp!]
\centering
\includegraphics[height=12cm]{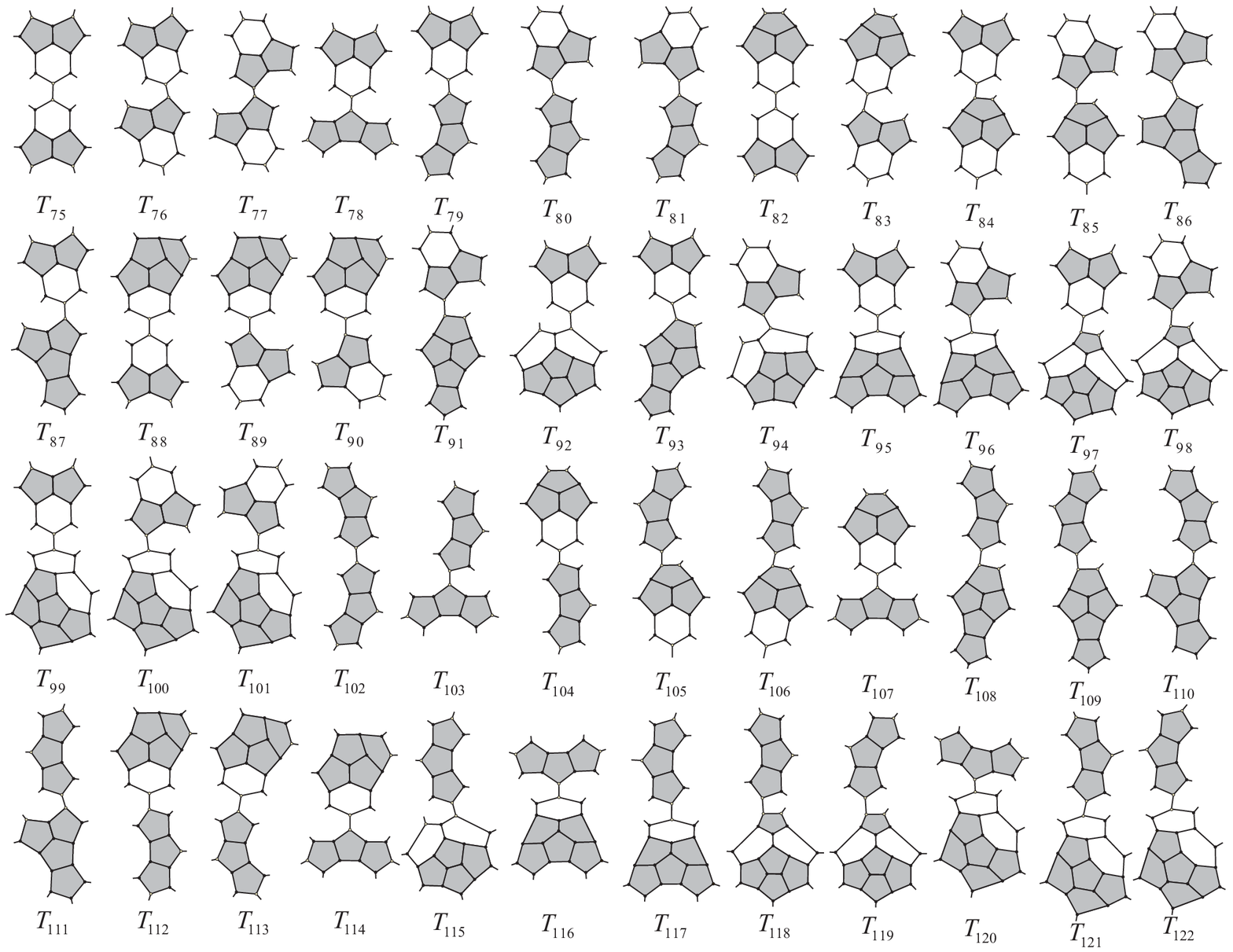}
\caption{\label{terminal12II}{\small $|\triangledown_F(F'')|=12$.}}
\end{figure}

$(O_6)$ If $k=12$, $a_i\in\{1, 2, 3, 4\}$ for $i\in\{1, 2, \ldots, 12\}$, and $a_j, a_{j+1}, a_{j+2}\in\{3, 4\}$, $a_{j-1}, a_{j+3}\in\{1, 2, 3\}$, then let the four half edges $t_j, t_{j+1}, t_{j+2}, t_{j+3}$ being incident with the four vertices $v_1, v_2, v_3,$ $v_4$ of a pentagon $P:=v_1\cdots v_5$. Add another new vertex $v$, a new edge $v_5v$, and attach two half edges to $v$.
A distance-array for the resulting generalized patch $L'$ would be $[a_1, \ldots, a_{j-2}, a_{j-1}+3, 1, a_{j+3}+3, a_{j+4}, \ldots, a_k]$.

$(O_7)$ If $k=12$, $a_i\in\{1, 2, 3, 4\}$ for $i\in\{1, 2, \ldots, 12\}$, then
we check that whether the twelve half edges of $L$ and the twelve half edges of some terminal generalized patch $T_r$, $r\in\{17, 18, \ldots, 180\}$ (see Fig. \ref{terminal12I}, \ref{terminal12II} and \ref{terminal12III})
can merge such that the resulting twelve new faces are pentagons or hexagons,
that is,
\begin{figure}[htbp!]
\centering
\includegraphics[height=21cm]{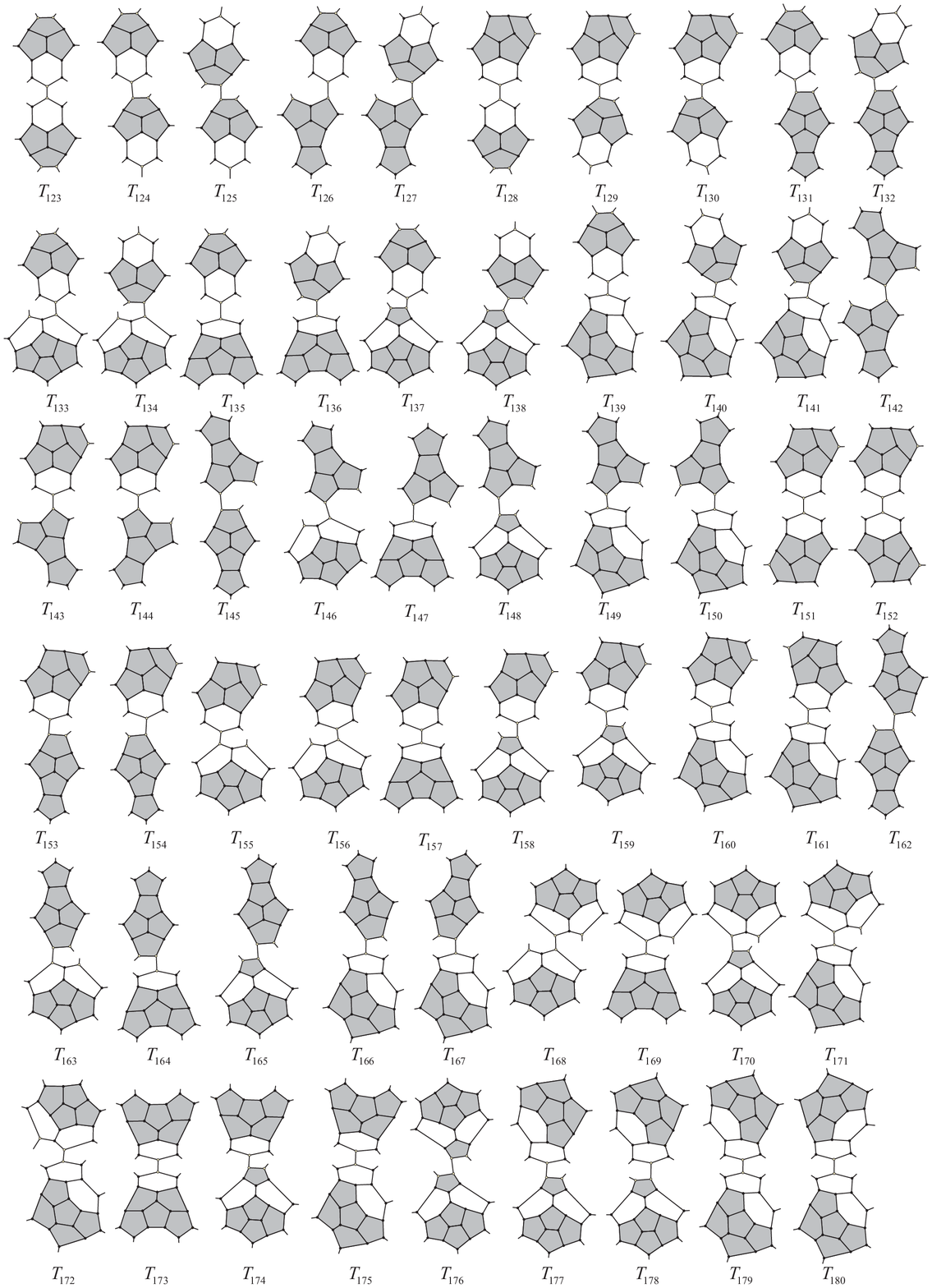}
\caption{\label{terminal12III}{\small $|\triangledown_F(F'')|=12$.}}
\end{figure}
we check that whether the sum $[a_{i_1}+b_1, \ldots, a_{i_{12}+b_{12}}]$ of some distance-array $[a_{i_1}, \ldots, a_{i_{12}}]$ of $L$ and the min-distance-array $[b_1, \ldots, b_{12}]$ of $T_r$, $r\in\{17, \ldots, 180\}$ has $a_{i_s}+b_s=5$ or $6$ for each $s=1, \ldots, 12$.
We record all efficient cases and merge the twelve edges according to the various efficient cases, respectively. Here, we may obtain different fullerenes.  However, the distance-array
of the resulting graph $L'$ is the empty distance-array $[]$.

We note that operation $O_1$ preserves the length of the distance-array. Operations $O_2$ and $O_6$ reduce the length of the distance-array by $2$. Clearly, the application of the operations $O_3, O_4, O_5$ and $O_7$ produces a fullerene graph, and no further operation can be performed.
So we call such operations the \emph{terminal operations}.
Set $F_t'$ be a subgraph obtained from some $L\in\mathcal{L}$ by performing $t\geq0$ times operation $O_1$.
Clearly, $F_t'$ is connected and $|\triangledown_F(F_t')|\leq12$. Let $F_t'':=F-F_t'$.
\begin{lem}\label{12_has_1-degree_vertex}
Suppose that $F_t''$ has a $1$-degree vertex and any $1$-degree vertex is not incident with two consecutive edges in $\triangledown_F(F_t')$ along the boundary of $F_t'$, then $|\triangledown_F(F_t')|=12$.
Moreover, $F_t''$ is two copies of $J_1$, and $\triangledown_F(F_t')$ has four consecutive edges along the boundary of $F_t'$ such that they are incident with four vertices of a pendent pentagon, respectively.
\end{lem}
\begin{proof}
We note that $F_t'$ is connected, $|\triangledown_F(F_t')|$ is an even number and $|\triangledown_F(F_t')|\leq12$.
Since $F_t''$ has not a $1$-degree vertex such that it is incident with two consecutive edges in $\triangledown_F(F_t')$ along the boundary of $F_t'$, $F_t''$ is not connected and has not an isolated edge.
By Lemma \ref{3-4-edge-cut}, $F_t''$ has exactly two connected components (say $B_1$ and $B_2$) and $|\triangledown_F(B_i)|=6$, $i=1, 2$ since $F_t''$ has a perfect matching and $|\triangledown_F(F_t')|\leq12$. Hence $|\triangledown_F(F_t')|=12$.
We claim that $B_i$ contains a cycle, $i=1, 2$. If not, $B_i$ is a path with $4$ vertices, then $F_t''$ has a $1$-degree vertex that is incident with two consecutive edges in $\triangledown_F(F_t')$ along the boundary of $F_t'$, a contradiction. So $\triangledown_F(B_i)$ is a cyclic $6$-edge-cut. Since $B_i$ has a unique perfect matching, by Proposition \ref{cyclic-6-edge-cut}, $B_i\cong J_1$, $i=1, 2$. This lemma holds since $F$ is a plane graph.
\end{proof}

We denote the min-distance-array of each $L_i\in\mathcal{L}$ by $L_i$, $i=1, \ldots, 95$. Then we have\\
$L_1=[12223125], L_2=[12232134], L_3=[12241224], L_4=[12331233], \ldots$,
$L_{95}=[12433335]$,
Starting from these $95$ initial min-distance-arrays, we describe the following procedure to generate a directed graph $D$ which is called the distance-array digraph.
\begin{alg}\label{alg}\emph{(}Generating the Distance-Array Digraph $D$\emph{)}
\begin{description}
\item[$(S1)$] Set $V=\{L_1, L_2, \ldots, L_{95}\}$ and $A=\emptyset$.
\item[$(S2)$] Select $L\in V$ on which the operations in $(S3)$ have not been made.
\item[$(S3)$] Implying all possible operations from $O_1$ to $O_7$ for the distance-array $L$, we obtain a set $R'$ of some distance-arrays. Replacing each distance-array in $R'$ with its min-distance-array, we obtain a set $R$.
    Set $V:=V\cup R$ and $A:=A\cup\{(L, L')| L'\in R\}$, where $(L, L')$ is an arc of $D$.
    Particularly, if $[]\in R$ and there are $t\geq1$ different ways to obtain $[]$ from $L$ \emph{(}note that we may have many generalized patches $T_i$ such that the edges in $\triangledown_F(L)$ and $\triangledown_F(T_i)$ can merge suitably, and for a generalized patch $T$, the edges in $\triangledown_F(L)$ and $\triangledown_F(T)$ may have different ways to merge suitably\emph{)}, then we use $t$ multiple arcs $(L, [])$ to represent those various ways to merge.
\item[$(S4)$] If all min-distance-arrays in $V$ have been selected to make operations in $(S3)$, then go to $(S5)$.
Otherwise, go to $(S2)$.
\item[$(S5)$] If $[]\notin V$, then $D:=\emptyset$. Otherwise, for every $L\in V$, delete $L$ if there is not directed path from $L$ to the empty distance-array $[]$. Then we obtain the final directed graph $D$ with vertex-set $V$ and arc-set $A$.
\end{description}
\end{alg}
\begin{figure}[htbp!]
\centering
\includegraphics[height=14cm]{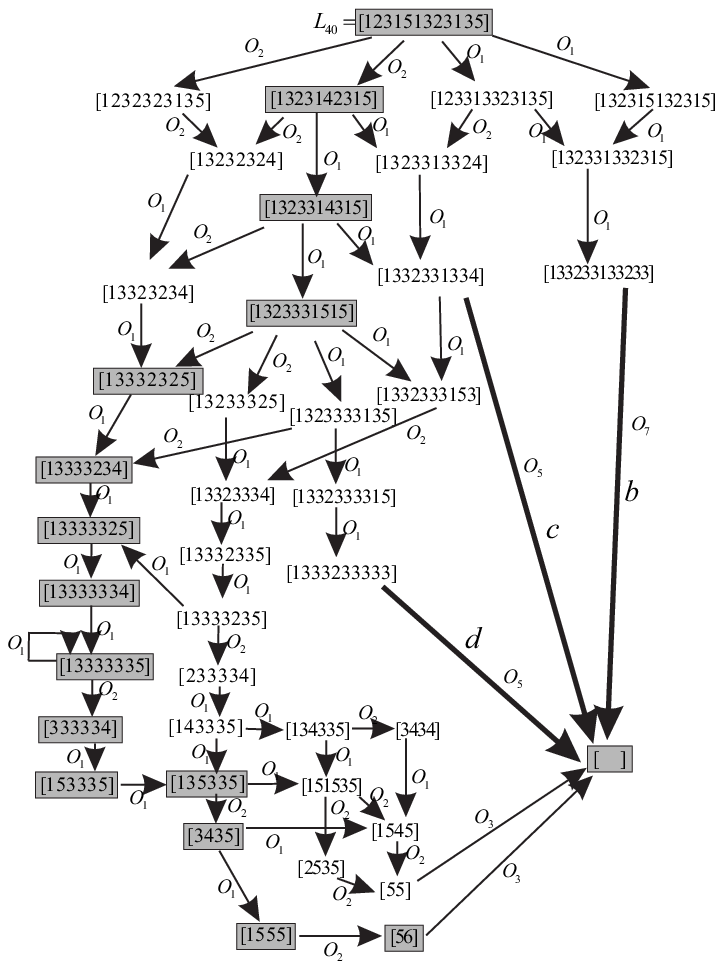}
\caption{\label{L40}{\small A subgraph $D'$ of $D$, the three dark arcs $b, c, d$ represent $11, 10, 13$ arcs, respectively.}}
\end{figure}
Using MATLAB we obtain that $D$ has $7802$ vertices and $28379$ arcs, and $\mathcal{L}^*\subset V(D)$, where $\mathcal{L}^*:=\mathcal{L}\setminus\{L_1, L_2, L_4, L_{14}, L_{44}, L_{45}, L_{70}, L_{75}\}$. In the sequel, we call $L_i\in\mathcal{L}^*$ the \emph{valid initial generalized patch} of $F$.
Starting from each valid initial generalized patch we can construct a fullerene with a forcing set of size $3$ by implying the seven operations $O_1$ to $O_7$. We note that $D$ has exactly four distance-arrays $[135151351515]$, $[1335135135]$, $[1333333335]$ and $[13333335]$, each of which is incident with a loop.

Since $D$ is too big, we take a subgraph of $D$ to demonstrate how to generate $D$. In the first step $(S1)$, we let $V=\{L_{40}\}$. By Algorithm \ref{alg}, we finally obtain the directed graph $D'$ (see Fig. \ref{L40}).
Let $F'$ be a generalized patch of $F$. If $F'$ has distance-array $[133233133233]$, then we can check that the twelve edges in $\triangledown_F(F')$ only can merge suitably with the twelve edges in $\triangledown_F(T_i)$, $i\in\{22, 27, 60, 67, 75, 78, 92, 97, 168, 170, 176\}$, and $\triangledown_F(F')$ and $\triangledown_F(T_i)$ has only one way to merge suitably. So the dark edge $b$ in Fig. \ref{L40} depicts eleven arcs from $[133233133233]$ to $[]$.
Similarly, if $F'$ has distance-array $[1332331334]$, then the ten edges in $\triangledown_F(F')$ only can merge suitably with the ten edges in $\triangledown_F(T_i)$, $i\in\{2, 5, 7, 13, 15\}$, and for each such $T_i$, $\triangledown_F(T_i)$ and $\triangledown_F(F')$ have two ways to merge suitably. So the dark edge $c$ in Fig. \ref{L40} depicts ten arcs from $[1332331334]$ to $[]$. For $F'$ with distance-array $[1333233333]$, the ten edges in $\triangledown_F(F')$ only can merge suitably with the ten edges in $\triangledown_F(T_i)$, $i\in\{1, 2, 3, 4, 5, 8, 9, 13, 15\}$. Checking each such $T_i$, we know that the dark edge $d$ in Fig. \ref{L40} depicts $13$ arcs from $[1333233333]$ to $[]$.

For example, along the directed walk $[123151323135]\rightarrow[1323142315]\rightarrow[1323314315]$\\
$\rightarrow[1323331515]\rightarrow[13332325]\rightarrow[13333234]\rightarrow[13333325]\rightarrow[13333334]\rightarrow[13333335]\rightarrow\cdots\rightarrow[13333335]\rightarrow[333334]\rightarrow[153335]\rightarrow[135335]\rightarrow[3435]\rightarrow[1555]\rightarrow[56]\rightarrow[]$
in $D'$ (there are $7k$ (integer $k\geq1$) times $O_1$ operations from $[13333335]$ to $[13333335]$), we can construct a nanotube fullerene with $k$-layers hexagons of type $(6, 1)$ (see Fig. \ref{nanotube(6,1)} (a)) and capped on the two ends by the two caps $A$ and $B$ (see Fig. \ref{nanotube(6,1)} (b), (c)), respectively. We can check that $\{e_1, e_2, e_3\}$ (see Fig. \ref{nanotube(6,1)} (b)) is a forcing set of this nanotube fullerene.
\begin{figure}[htbp!]
\centering
\includegraphics[height=3.3cm]{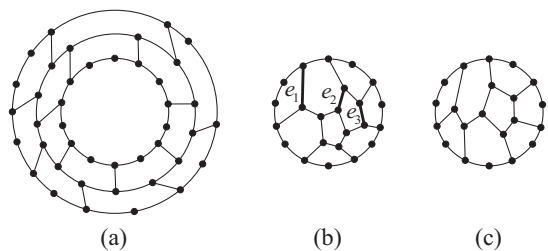}
\caption{\label{nanotube(6,1)}{\small (a) $2$-layers hexagons of type $(6, 1)$, (b) cap $A$, (c) cap $B$.}}
\end{figure}

We note that although $F_{24}$ (see Fig. \ref{F_24}) has the minimum forcing number $2$, it also has a minimal forcing set of size $3$, where minimal means that its any proper subset is not a forcing set any more. So $F_{24}$ can also be generated from a forcing set of size $3$.

In the following, we will prove that $D$ represents all possible ways to
obtain fullerene graphs from the initial generalized patches in $\mathcal{L}$ by implying the seven operations $O_1$ to $O_7$ (in any order).
\begin{thm}\label{main}
A fullerene graph has the minimum forcing number $3$ if and only if it is isomorphic to some $W_i$, $i\in\{1, \ldots, 81\}$ \emph{(}see Fig. \emph{\ref{fullerene1}, \ref{fullerene2})}, or its order is not $24$ and can be constructed from some valid initial generalized patches by implementing operations $O_1$ to $O_7$ along a directed walk in $D$ from its min-distance-array to $[]$.
\end{thm}
\begin{proof}
By Lemma \ref{(4,2)-nanotube}, a nanotube fullerene of type $(4, 2)$ has a forcing set $S$ of size $3$ such that $F[V(S)]\cong L_3$ (see Fig. \ref{seed-case1}). We note that $L_3\in D$. Further, we can check that a nanotube fullerene of type $(4, 2)$ can be constructed from $L_3$ by applying several times operations $O_1, O_2$ and $O_3$.

Next, we suppose that $F$ is a fullerene with the minimum forcing number $3$ and is not a nanotube fullerene of type $(4, 2)$. Then $F$ has a forcing set $S:=\{e_1, e_2, e_3\}$.
By Theorem \ref{initial_seed} and the above discussions, $F$ is isomorphic to some $W_i$, $i\in\{1, \ldots, 81\}$ or
$F$ has a generalized patch $L\in\mathcal{L}$ such that $S\subseteq E(L)$ forces a perfect matching of $L$ in the sense of $F$.
Since $S$ is a forcing set of $F$, $F-L$ has a unique perfect matching. Clearly, $|\triangledown_F(L)|\leq12$.
If $F-L=\emptyset$, then $L$ is not an induced subgraph of $F$, that is, there exits $u, v$ on the boundary of $L$ such that $uv\in E(F)\setminus E(L)$. We can show that operation $O_2$ or $O_3$ can be done.
Next, we suppose that $F-L\neq\emptyset$. It implies that $|\triangledown_F(L)|\geq4$.

\emph{Case 1.} $L$ is not an induced subgraph of $F$.

Then there exits two vertices on the boundary of $L$ such that they are adjacent in $F$ and are not adjacent in $L$.
So $|\triangledown_F(L)|\leq10$.
If there are two such vertices $u, v\in V(L)$ on the boundary of $L$ such that the half edge incident with $u$ and the half edge incident with $v$ are consecutive along the boundary of $L$, then operation $O_2$ can be done. Otherwise, $F-L$ is not connected. Since $F-L$ has a perfect matching and $|\triangledown_F(L)|\leq10$, $F-L$ has exactly two components, denoted by $H_1$ and $H_2$.
Clearly, $|\triangledown_F(H_i)|$ is even and $|\triangledown_F(H_i)|\geq4$.
Since $|\triangledown_F(H_1)|+|\triangledown_F(H_2)|=|\triangledown_F(L)|\leq10$, there is some $i\in\{1, 2\}$ with $|\triangledown_F(H_i)|=4$. Without loss of generality, we suppose that $|\triangledown_F(H_1)|=4$. Hence $H_1$ is an edge by Lemma \ref{3-4-edge-cut}, that is, $F-L$ has a $1$-degree vertex that is incident with two consecutive edges in $\triangledown_F(L)$ along the boundary of $L$. So operation $O_1$ can be done.

\emph{Case 2.} $L$ is an induced subgraph of $F$.

We consider the following two cases.

\emph{Subcase 2.1.} $F-L$ has a $1$-degree vertex.

If $F-L$ has a $1$-degree vertex that is incident with two consecutive edges in $\triangledown_F(L)$ along the boundary of $L$, then operation $O_1$ can be done.
If $F-L$ has a $1$-degree vertex and any such vertex is not incident with two consecutive edges in $\triangledown_F(L)$ along the boundary of $L$, then $|\triangledown_F(L)|=12$ and $F-L\cong 2J_1$ by Lemma \ref{12_has_1-degree_vertex}. So operation $O_6$ may be done.

\emph{Subcase 2.2.} $F-L$ has no $1$-degree vertices.

By Lemma \ref{3-4-edge-cut} and Proposition \ref{cyclic-6-edge-cut}, $F-L$ is connected.
Moreover, by Theorem \ref{Operation3} and Corollary \ref{terminal-gene-patch},
$F-L$ is isomorphic to $J_2$ if $|\triangledown_F(L)|=8$, and is isomorphic to $PP$ (see Fig. \ref{sub-patches}) or some $T_i$, $i\in\{1, \ldots, 16\}$ (see Fig. \ref{terminal10}) if $|\triangledown_F(L)|=10$, and is isomorphic to some $T_i$, $i\in\{17, \ldots, 180\}$ (see Fig. \ref{terminal12I}, \ref{terminal12II} and \ref{terminal12III}) if $F-L$ has not pendent pentagons and $|\triangledown_F(L)|=12$. So operations $O_4, O_5, O_7$ can be done.
If $F-L$ has a pendent pentagons and $|\triangledown_F(L)|=12$, then operation $O_6$ can be done.

Let $L^t$ ($t\geq1$) be the subgraph of $F$ obtained from $L$ by implementing $t$ times operations $O_1$, $O_2$, $\ldots$, or $O_7$.
If the distance-array of $L^t$ is $[]$, then we are done.
Next, we suppose that the distance-array of $L^t$ is not $[]$. It means that $L^t$ is obtained from $L$ by implementing $t$ times operations $O_1, O_2$, or $O_6$.
By the definition of the operations $O_1, O_2$, and $O_6$, $F-L^t$ has a unique perfect matching and $|\triangledown_F(L^t)|\leq12$.
We note that if operation $O_6$ is implemented at least once in the process of obtaining $L^t$ from $L$,
then $|\triangledown_F(L^t)|\leq|\triangledown_F(L)|-2$, and
the conclusions of Theorem \ref{Operation3} and Corollary \ref{terminal-gene-patch} also hold for $|\triangledown_F(L^t)|=8$ or $10$ with $F-L^t$ having no $1$-degree vertices.
As the discussion of $L$,
we can show that at least one of the operations $O_1$, $O_2$, $\ldots$,  $O_7$ can be done for $L^t$.

Since $F$ is finite, there is an integer $t_0\geq1$ such that the distance-array of $L^{t_0}$ is $[]$.
Hence $L\in D$ and $F$ can be constructed from valid initial generalized patch $L$ by implementing operations $O_1$ to $O_7$ along a directed walk in $D$ from $L$ to $[]$.

Conversely, each fullerene $W_i$, $i\in\{1, \ldots, 81\}$ has the minimum forcing number three. It is sufficient to consider a fullerene $F$ constructed from some valid initial generalized patch $L_i$ by implementing operations $O_1$ to $O_7$ along a directed walk in $D$ from $L_i$ to $[]$. Clearly, $L_i$ has an edge set $S:=\{e_1, e_2, e_3\}$ such that $S$ forces a perfect matching of $L_i$ in the sense of $F$. By the description of the operations $O_1$ to $O_7$, $S$ is a forcing set of $F$. Since the order of $F$ is not $24$, $F\ncong F_{24}$. So $f(F)=3$.
\end{proof}

We notice that there is only one fullerene $F_{24}$ (see Fig. \ref{F_24}) of order $24$.
\begin{cor}\label{anti-forcing3}
Except for fullerene $F_{24}$, each fullerene with anti-forcing number $4$ has the minimum forcing number $3$.
\end{cor}
\begin{proof}
Let $F$ be a fullerene with anti-forcing number $4$ and $F\ncong F_{24}$. Then the order of $F$ is not $24$.
From the Theorem $4.3$ in Ref. \cite{YangQ}, $F$ can be constructed from $L_5, L_{26}$, or $L_{27}$ by implying operations $O_1$ to $O_4$.
So $F$ has the minimum forcing number $3$ by Theorem \ref{main}.
\end{proof}
We note that a nanotube fullerene of type $(4, 2)$ has anti-forcing number $4$ and minimum forcing number $3$.
By Theorem $4.4$ in Ref \cite{YangQ} and Corollary \ref{anti-forcing3}, the following corollary holds.
\begin{cor}
For any even $n\geq20$ $(n\neq22, 24)$, there is a fullerene $F$ of order $n$ such that $F$ has the minimum forcing number $3$.
\end{cor}


\end{document}